\documentclass[11pt,letterpaper]{amsart}
\usepackage{amsfonts,amssymb,amsmath,amsgen,amsthm}
\usepackage{hyperref}

\usepackage{color}

\def\d{{\partial}}
\def\eps{\varepsilon}

\def\R{{\mathbb R}}
\def\N{{\mathbb N}}
\def\T{{\mathbb T}}

\def\O{\mathcal O}

\def\dd{\mathrm d}

\def\si{\sigma}

\def\IM{\operatorname{Im}}
\def\RE{\operatorname{Re}}
\DeclareMathOperator{\diver}{div}
\DeclareMathOperator{\supp}{supp}

\def\uphi{{\breve \phi}}
\def\ua{{\breve a}}

\def\Tend#1#2{\mathop{\longrightarrow}\limits_{#1\rightarrow#2}}

\def\<{{\langle}}
\def\>{{\rangle}}
\def\({\left(}
\def\){\right)}

\theoremstyle{plain}
\newtheorem{theorem}{Theorem}[section]

\newtheorem{lemma}[theorem]{Lemma}

\newtheorem{proposition}[theorem]{Proposition}

\theoremstyle{remark}
\newtheorem{remark}[theorem]{Remark}

\numberwithin{equation}{section}

\begin{document}
\title[Pathological set with loss for NLS]{Pathological set with loss
  of regularity for nonlinear 
  Schr\"odinger equations}
\address{Univ Rennes, CNRS\\ IRMAR - UMR 6625\\ F-35000
  Rennes, France}
\author[R. Carles]{R\'emi Carles}
\email{Remi.Carles@math.cnrs.fr}

\author[L. Gassot]{Louise Gassot}
\email{louise.gassot@univ-rennes1.fr}

\begin{abstract}
  We consider the mass-supercritical, defocusing, nonlinear
  Schr\"odinger equation. We prove  loss of regularity
 in arbitrarily short times  for regularized initial data belonging to a dense set of any fixed
  Sobolev space for which the nonlinearity is supercritical. The proof
  relies on the construction of initial data as a superposition of
  disjoint  bubbles at different scales. We get an approximate 
  solution with a time of existence bounded from below, provided by the compressible Euler equation, which enjoys
  zero speed of propagation. Introducing suitable renormalized modulated energy
  functionals, we prove spatially localized estimates
  which make it possible to obtain the loss of regularity. 
\end{abstract}

\thanks{The authors are supported by Centre Henri Lebesgue,
  program ANR-11-LABX-0020-0. A CC-BY public copyright license has been applied by the authors to the present document and will be applied to all subsequent versions up to the Author Accepted Manuscript arising from this submission.}

\maketitle

\section{Introduction}
\label{sec:intro}

We consider the defocusing nonlinear Schr\"odinger equation, for $m\in \N$:
\begin{equation}
  \label{eq:nls}
  i\d_t \psi+\frac{1}{2}\Delta \psi = |\psi|^{2m}\psi,\quad \psi_{\mid t=0}=f_0,
\end{equation}
where $x\in \R^d$ and $f_0 \in H^s(\R^d)$. We assume that the
nonlinearity is $L^2$-supercritical,
\begin{equation*}
  s_c=\frac{d}{2}-\frac{1}{m}>0,
\end{equation*}
and we suppose that $0<s<s_c$: the nonlinearity is $H^s$-supercritical.
For technical reasons, we also assume $s\le 2$ (which is no further
restriction when $d\le 4$). Our goal is to
improve~\cite[Theorem~1.4]{CaGa23}, recalled below, and prove some
loss on regularity in the spirit of \cite{Lebeau05}. We emphasize the
fact that our analysis remains valid for compact geometries, typically
for \eqref{eq:nls} on the torus $\T^d$.


\subsection{Context}
It is known that when $s\geq s_c$, then the Cauchy
problem~\eqref{eq:nls} is locally well-posed in $H^s(\R^d)$
(see~\cite{CW90}), whereas when $s<s_c$, then the Cauchy problem is
ill-posed, as established initially in \cite{KPV01}. In \cite{CCT2},
the notion of norm inflation was introduced, and proven in the case of
\eqref{eq:nls}: there exists a sequence of initial data
$(f^h_0)_{h\in(0,1]}$ in $ \mathcal{S}(\R^d)$ going to zero in $H^s$,
but such that the corresponding maximal solutions
$(\psi^h)_{h\in(0,1]}$ are defined on $[0,t^h]$ for some $t^h\to 0$
and $\psi^h(t^h)$ goes to infinity in $H^s$: 
\[
\|f^h_0\|_{H^s}\Tend{h}{0} 0,
\quad
\|\psi^h(t^h)\|_{H^s}\Tend{h}{0} +\infty.
\]
It turns out that this norm inflation mechanism also occurs around any
initial data, as proven initially in \cite{SuTz20a} in the case of the
wave equation (see also \cite{Tz19,Xia21}), and more recently in~\cite{Xia22}
for a fourth-order Schrödinger equation. 

The question of norm inflation becomes more delicate
when replacing the
sequence $(f_0^h)_h$ of initial data by the sequence $(\iota_h\ast
f_0)_h$, where the approximate identity $\iota_h$ is given by
\[
\iota_{h}(x)=\frac{1}{h^d}\iota\left(\frac{x}{h}\right),
\]
with $\iota\in\mathcal{C}_{c}^{\infty}(\R^d)$ and $\int_{\R^d}\iota(x)\dd x=1$.
In this case, it has been shown in~\cite[Theorem~1.4]{CaGa23}, in the case $d=3$, that there exists a dense set of functions $f_0\in H^s(\R^3)$ (called {\it pathological set}) such that norm inflation holds for this sequence of regularized initial data:
\[
\|\iota_{h_k}\ast f_0-f_0\|_{H^s}\Tend{k}{\infty} 0,
\quad
\|\psi^{h_k}(t^{h_k})\|_{H^s}\Tend{k}{\infty} +\infty.
\]
The construction of such a pathological set was first evidenced in
\cite{SuTz20} for the wave equation, then extended to Schr\"odinger
equations in~\cite{CaGa23} by removing a finite speed of propagation
argument. The result is also valid on $\T^3$.
\smallbreak

On the other hand, ill-posedness for the range of exponents $s<s_c$
was strengthened in~\cite{ACMA} as a loss of regularity result
(extending the cubic case from \cite{CaARMA}; see also
\cite{ThomannAnalytic}). More precisely, there exists a sequence of
initial data $(f^h_0)_{h\in(0,1]}\subset \mathcal{S}(\R^d)$,  global
weak solutions $(\psi^h)_{h\in(0,1]}$, and  $t^h\to 0$, such that:
\[
\|f^h_0\|_{H^s}\Tend{h}{0} 0,
\quad
\|\psi^h(t^h)\|_{H^{\si}}\Tend{h}{0} +\infty, \text{ for all }
\si>\frac{s}{1+m(s_c-s)}.
\]
This result is an analogue of the original loss of regularity theorem
from Lebeau~\cite{Lebeau05} concerning energy-supercritical wave
equations ($s_c>1$),
\[
(\partial_{tt}-\Delta)u+u^{2m+1}=0,
\]
which is as follows. There exists $f_0$ in $H^s$ and $t^{h_k}\to 0$ such that the solution $\psi$ to the wave equation satisfies that for every $\si > I(s)$, 
\[
\|\psi(t^{h_k})\|_{H^{\si}}
	\Tend{k}{+\infty} +\infty.
\]
The exponent $I(s)$ is given by $I(s)=\frac{s}{1+m(s_c-s)}$ when $s\ge
s_{\rm sob}$,
\[
  s_{\rm sob} =\frac{dm}{2m+2}\text{ being such that }\dot
  H^{s_{\rm sob}}(\R^d)\hookrightarrow L^{2m+2}(\R^d),
\]
and $I(s)=1$ when $s\le s_{\rm
  sob}$ (note that $I(s_{\rm sob})=1$). The result from
\cite{Lebeau05} uses in a crucial fashion the 
property that (weak) solutions to the nonlinear wave equation enjoy
finite speed of propagation. It is mostly because of this that the
result of \cite{ACMA} concerns a sequence of initial data (one bubble)
rather than some fixed data (superposition of disjoint bubbles) like
in \cite{Lebeau05}. In this paper, we prove a result which is
essentially the same as in \cite{Lebeau05}, by looking at the sequence
of initial data $(\iota_h\ast f_0)_h$ regularized by convolution. Of
course there is no finite speed of propagation for \eqref{eq:nls}. Instead,
our argument takes advantage of a finite propagation speed property at the
level of compressible Euler equations, which naturally appears in the
WKB analysis of the semi-classical version of \eqref{eq:nls}. 

\subsection{Main results}
\color{black}
Our first result is in the spirit of \cite[Theorem~1.33]{Tz19}, and
states that the main result from \cite{ACMA} is valid not only at the
origin, but near any initial datum in $H^s(\R^d)$:
\begin{theorem}[Loss of regularity near any initial datum]\label{thm:0}
  Let $0<s<s_c$, with $s\le 2$. For any $f_0\in H^s(\R^d)$, there
  exists a sequence $f_{0,k}\in \mathcal C_c^\infty(\R^d)$ and times
  $t_k\to 0$ as $k\to \infty$, such that
  \begin{equation*}
    \|f_0-f_{0,k}\|_{H^s} \Tend{k}{\infty} 0,
\quad
\|\psi_k(t_k)\|_{H^{\si}}\Tend{k}{\infty} +\infty,\quad \forall
\si>\frac{s}{1+m(s_c-s)}, 
\end{equation*}
where for any (fixed) $k$, $\psi_k\in L^\infty(\R;H^1\cap L^{2m+2})$
solves~\eqref{eq:nls} with initial datum $f_{0,k}$.
\end{theorem}
As a corollary of this result, in the energy-supercritical case
$s_c>1$, taking $s=s_{\rm sob}$ yields $f_{0,k}$ converging to
$f_0$ in the energy space $H^1\cap L^{2m+2}$, while $\psi_k$ is
instantaneously unbounded in $H^\si$ for any $\si>1$, since $I(s_{\rm
  sob})=1$. 
Like in \cite{SuTz20}, one may ask how the above phenomenon depends on
the approximating sequence $(f_{0,k})_k$, and consider a general
approximate identity. Our main result is as follows, and somehow
unites the results from \cite{ACMA,CaARMA,ThomannAnalytic} on the one
hand, and \cite{CaGa23} on the other hand:
\color{black}
\begin{theorem}[Loss of regularity for regularized data]\label{thm:main}
Let $0<s<s_c$, with $s\le 2$. There exists a dense pathological set of initial data $f_0$ in $H^{s}(\R^d)$ such that the following holds.
For $h>0$, let $\psi^h$ be the solution to~\eqref{eq:nls} with initial data $\iota_{h}\ast f_0$.
There exist a sequence of parameters $h_k\to 0$ and times $t_k\to
0$ as $k\to\infty$, such that 
\[
\|\iota_{h_k}\ast f_0-f_0\|_{H^s}\Tend{k}{\infty} 0,
\quad
\|\psi^{h_k}(t_k)\|_{H^{\si}}\Tend{k}{\infty} +\infty,\quad \forall \si>\frac{s}{1+m(s_c-s)}.
\]
\end{theorem}

\begin{remark}
We shall see in the proof that this result is also valid in $\T^d$.
\end{remark}

\begin{remark}
  Note that as mentioned in~\cite[Theorem~1.4 and Proposition 2.10]{CaGa23}, as soon as the Cauchy problem is globally well-posed in $H^k(\R^d)$ for some $k>\frac{d}{2}$, then the set of initial data $f_0$ satisfying Theorem~\ref{thm:main} contains a dense $G_{\delta}$ set.
 \end{remark}

The pathological set of initial data, on which norm inflation or loss
of regularity happens, is the counterpart of the set of initial data
such that probabilistic well-posedness holds, initiated by
Bourgain~\cite{Bou94,Bou96} for the cubic Schrödinger equation on
$\T^2$, then developed by Burq and Tzvetkov~\cite{BT08a,BT08b} for the
cubic wave equation on manifolds, \textcolor{black}{ and by Burq,
  Thomann and Tzvetkov \cite{BuThTz13} for the nonlinear Schr\"odinger
  equation on $\R$}. Indeed, as mentioned
in~\cite[Theorem~1.3]{CaGa23}, the generic well-posedness result from
B\'enyi, Oh and Pocovnicu \cite{BeOhPo15} on the cubic Schr\"odinger
equation on $\R^3$ implies that for $\frac 14<s<s_c$, there exists a
probability measure $\mu$ supported on $H^s(\R^3)$ and a dense set
$\Sigma\subset H^s(\R^3)$ with full $\mu$-measure such that the
following holds.  For every $f_0\in \Sigma$, the solution $\psi^h$
to~\eqref{eq:nls} with initial data $\iota_h\ast f_0$ is well-defined
up to some time $T(f_0)$ and converge to some limiting distributional
solution $\psi$ to~\eqref{eq:nls} with initial data $f_0$ on
$[0,T(f_0)]$: 
\[
\|\iota_{h}\ast f_0-f_0\|_{H^s}\Tend{h}{0} 0,
\quad
\|\psi^h-\psi\|_{L^\infty([0,T],H^s(\R)^3)}\Tend{h}{0} 0.
\]
This result was improved in \cite{BeOhPo19}, with the
lower bound $s>\frac15$, \textcolor{black}{ and in
  \cite{Spitz23} with the lower bound $s>\frac17$}. 
Consequently, the full-measure set $\Sigma$ must be disjoint from the
pathological set. In particular, if the pathological set contains a
dense $G_{\delta}$ set, this cannot be the case for
$\Sigma$. \textcolor{black}{This conflict between full measure
  (probability one for initial data) and density (of the pathological
  set) is also striking in view of the results of \cite{KiMuVi19}
  (see also \cite{DoLuMe19,Spitz23} where the initial Sobolev
  regularity is lowered), since for the cubic Schr\"odinger equation
  on $\R^4$ (which is an energy-critical equation), it is proved that
  for all initial data in $H^s(\R^4)$ with $\frac{5}{6}<s<1$,
  the nonlinear evolution of the randomization of $f$ is almost surely
  global in time and stable in the sense that it is asymptotically
  linear (scattering).}

\begin{remark}
 For $d\ge 5$, it was proven in \cite{MRRS22} that for
certain values of $m$ at least, concerning energy-supercritical cases
($s_c>1$), there exists a finite co-dimensional manifold of smooth
initial data with spherical symmetry such that the solution blows up
in finite time $T_*>0$, with
\begin{equation*}
  \|\psi(t)\|_{L^\infty(\R^d)}\Tend t {T_*}+\infty. 
\end{equation*}
A common feature of our analysis with the approach in \cite{MRRS22} is
the use of hydrodynamical formulations, measuring a strong interaction
between the phase and the amplitude for the solution of
\eqref{eq:nls}, an aspect already at the origin of the results in
\cite{CaARMA}. In particular, it is proven in \cite{MRRS22} that there
exists $1<\si<s_c$ such that
\begin{equation*}
  \|\psi(t)\|_{H^\si(\R^d)}\Tend t {T_*}+\infty. 
\end{equation*}
Our main result concerns the whole range $\si>\frac{s}{1+m(s_c-s)}$,
which, for $s_c>1$ and $s=s_{\rm sob}$, corresponds exactly to
$\si>1$. 
\end{remark}

\subsection{Scheme of the proof}
\label{sec:scheme}

The general strategy mixes ideas from \cite{ACMA} and
\cite{CaGa23}. In \cite{ACMA}, and like in \cite{BGTENS,CCT2}, the general
data are of the form 
\begin{equation*}
  f_0^h(x) = h^{s-d/2}a\(\frac{x}{h}\).
\end{equation*}
Introducing $\eps=h^{m(s_c-s)}$ and the change of unknown function
\begin{equation*}
  u^\eps(t,x) = h^{d/2-s}\psi^h\(h^2\eps t,hx\),
\end{equation*}
the family of function $(u^\eps)_{0<\eps\le 1}$ solves the
semiclassical version of \eqref{eq:nls},
\begin{equation*}
  i\eps\d_t u^\eps+\frac{\eps^2}{2}\Delta u^\eps =
  |u^\eps|^{2m}u^\eps,
\end{equation*}
with initial data $u^\eps_{\mid t=0}=a$ independent of $\eps$. The
solution $u^\eps$ becomes instantaneously $\eps$-oscillatory, in the
sense that there exists $\tau>0$ (independent of $\eps$) such that
\begin{equation*}
  \|u^\eps(\tau)\|_{\dot H^1}\approx \frac{1}{\eps}. 
\end{equation*}
In principle, this property can be shown thanks to WKB analysis, as in
\cite{CaARMA} for the cubic case $m=1$, and in
\cite{ThomannAnalytic} for analytic data. However the proof
in \cite{ACMA} is cheaper in the sense that it merely requires the use
of modulated energy functionals, without justifying WKB
analysis. Using various interpolation estimates and the property
$\|u^\eps(\tau)\|_{L^2}=\|a\|_{L^2}$ (independent of $\eps$), one
infers
\begin{equation*}
  \|u^\eps(\tau)\|_{\dot H^\si}\approx \frac{1}{\eps^\si},\quad
  \forall \si\ge 0.  
\end{equation*}
The result follows when using the scaling to go back to
$\psi^h$, with $t^h= h^2\eps\tau$.
\smallbreak

\color{black}
The proof of Theorem~\ref{thm:0} can be viewed as a consequence of the
proof of Theorem~\ref{thm:main}, as we explain in
Appendix~\ref{sec:th0}, so we now focus on the proof of the latter.
\color{black}
Instead of starting from one concentrating data, we start from a
superposition of such bubbles, like in \cite{Lebeau05} initially,
along a sequence $h_k\to 0$ as $k\to +\infty$ (see
Section~\ref{sec:set}). As pointed out above, 
we also regularize this sum of bubbles, like in \cite{CaGa23}. We then
adapt the modulated energy analysis along the bubble corresponding to
the scale $h_k$. The more direct approach is interesting under the constraint
$s< s_{\rm sob}$  (Section~\ref{sec:modulated}). It is
improved in Section~\ref{sec:renormalized} by considering a
renormalized modulated energy. However, unlike in the case of a single bubble
recalled above, this is not enough to conclude directly, as the
$L^2$-norm of the rescaled function $u^\eps$ is not uniformly bounded,
due to the (initial) bubbles corresponding to $\ell<k$.  In the spirit
of \cite{Lebeau05}, we pick bubbles 
with initial pairwise disjoint supports.
Using a finite propagation speed for the approximate bubbles
involved in the modulated energy functionals, due to Makino, Ukai and
Kawashima \cite{MUK86} (see Section~\ref{sec:bubble}), we prove
spatially localized estimates in Section~\ref{sec:local}, and then
conclude. In an appendix, we provide an alternative proof of the most
delicate result of Section~\ref{sec:local}, which does not use Fourier
analysis, and is thus more flexible for different geometries.

\section{Pathological set of initial data}
\label{sec:set}

\subsection{Definition}\label{sec:def-set}
Following~\cite{SuTz20,CaGa23}, we consider pathological initial data as a superposition of bubbles displaying norm inflation at different scales, of the form
\begin{equation*}
  f_ 0 =\varphi_0+\sum_{k=k_0}^\infty \varphi_k,
\end{equation*}
for some $k_0\ge 1$. 
The background $\varphi_0\in\mathcal C_c^{\infty}(\R^d)$ automatically satisfies that $\R^d\setminus \operatorname{supp}(\varphi_0)$ is a large set, in particular it contains a ball $\overline{B}(0,r_0)$. In order to extend our result to $\T^d$, one should assume that  $\T^d\setminus \operatorname{supp}(\varphi_0)$ contains an open set $B(0,r_0)$. As such, we will choose bubbles $\varphi_k$ with pairwise disjoint supports inside of $\overline{B}(0,r_0)$.

We fix some parameter $M>1$, and define the scale
\[
h_k:=e^{-M^k}.
\]
Each bubble $\varphi_k$ is a rescaling of a profile
$\alpha_k\in\mathcal{C}_c^{\infty}(\R^d)$ at scale $h_k>0$:
\color{black}
\begin{equation}\label{eq:def-bubble}
  \varphi_k(x) =
 \frac{1}{ |\log(h_k)|}h_k^{s-\frac{d}{2}}\alpha_k\(\frac{x}{h_k}\).
\end{equation}
\color{black}
We assume for simplicity that 
\[
\alpha_k(x)=\alpha\left(x-\frac{x_k}{h_k}\right) 
\]
for some fixed profile $\alpha\in\mathcal{C}_c^{\infty}(\R^d)$, where
the position $x_k\in\R^d$ will be characterized later on.
\color{black}
The radial
case can be handled by considering instead, for instance,
$\alpha_k(x)=\alpha(|x|^2-r_k^2/h_k^2)$ with $\alpha\in \mathcal C_c^\infty(\R_+)$.
\color{black}
\color{black}
The logarithmic factor aims at guaranteeing the convergence of the above
series in $H^s$, but its presence can be forgotten to grasp the main
ideas and details of the computations.
\color{black}
Let $\psi_k$ solve~\eqref{eq:nls} with regularized initial data
defined from $f_0$ as 
\begin{equation*}
  {\psi_k}_{\mid t=0}:=  \iota_{h_k/100}\ast f_0.
\end{equation*}
\color{black}
We note that for any fixed $k$, we can find such a global in time solution,
$\psi_k\in L^\infty(\R;H^1\cap L^{2m+2})$, as a weak solution, from
\cite{GV85c}. If $s_c\le 1$, $\psi_k$ is actually the unique, global, mild
solution, and it is smooth for all time by propagation of
regularity. If $s_c>1$, as there exists a unique local solution to~\eqref{eq:nls} in $H^N(\R^d)$ for $N>d/2$ by standard arguments,  at
least locally in time, the weak solution is the unique, smooth,
solution. Note that this means that for any $k$, $\psi_k$ remains
smooth locally in time, on a time interval which may shrink to
$\{0\}$ as $k$ goes to infinity, in agreement with the result from \cite{MRRS22}.  
\color{black}
We expect that this regularized solution will display a norm inflation in $H^{\si}$ around the $k$-th bubble at some time $t_k$ to be defined in the next paragraph.

\subsection{Semiclassical form}\label{sec:rescale}

In order to show a loss of regularity result in the spirit of Lebeau~\cite{Lebeau05}, we rather consider the rescaled equation from~\cite{ACMA} following
\cite{CaARMA}: consider
\color{black}
 \begin{equation}\label{eq:epsilon}
\eps_k=h_k^{m(s_c-s)} |\log(h_k)|^{m}\Tend k \infty 0,
\end{equation}
\color{black}
where this limit stems from the assumption $s<s_c$. Rescale the
function $\psi$ as
\color{black}
\begin{equation*}
u_k (t,x) 
	=  |\log(h_k)| h_k^{\frac{d}{2}-s} \psi_k \( h_k^2 \eps_k t,h_k x\).
 \end{equation*}
 \color{black}
The equation satisfied by $u_k$ is the semiclassical Schr\"odinger equation
\begin{equation}\label{eq:NLSsemi}
i\eps_k \d_t u_k +\frac{\eps_k^2}{2} \Delta u_k
	= \left\lvert  u_k\right\rvert^{2m}u_k
	\quad ;\quad
	{u_k}_{\mid t=0}=u_{0,k}. 
\end{equation}
According to the definition, the initial data $u_{0,k}$ is given by the following formula.

\begin{lemma}[Rescaled initial data]\label{lem:rescalingCI}
One can write 
\begin{equation*}
  u_{0,k}(x) = \jmath\ast\(\varphi_{0,k} +\sum_{\ell=k_0}^{\infty}
   \varphi_{\ell,k}\),
\end{equation*}
where $\jmath=\iota_{1/100}$,
\color{black}
\[
 \varphi_{0,k} (x) =
           |\log(h_k)| h_k^{d/2-s}\varphi_0(h_k x),
\]
and for $\ell\geq k_0$,
\begin{equation*}
  \varphi_{\ell,k}(x) = 
	\frac{ |\log(h_k)|}{ |\log(h_\ell)|}\(\frac{h_k}{h_\ell}\)^{d/2-s}  \alpha_\ell\(\frac{h_k}{h_\ell}x\).
  \end{equation*}
  \color{black}
\end{lemma}

\begin{proof}
We have
\begin{equation*}
  u_{0,k}(x) = \mathcal R_k \(\iota_{h_k/100}\ast f_0\),
\end{equation*}
where $\mathcal R_k $ is the scaling transformation
\begin{equation*}
  \mathcal R_k f(x) =  |\log(h_k)| h_k^{d/2-s}f(h_k x). 
\end{equation*}

We first write this initial data in a more convenient way. When $k\geq 1$, $\ell\geq k_0$ and $h>0$ we have
\begin{multline*}
  \mathcal R_k \( \iota_{h}\ast \varphi_\ell\)(x)
  	=	\frac{ |\log(h_k)|}{ |\log(h_\ell)|} \(\frac{h_k}{h_\ell}\)^{d/2-s} 
	\int \frac{1}{h^d}\iota\(\frac{y}{h}\) \alpha_\ell\(\frac{h_k}{h_\ell}x-\frac{y}{h_\ell}\)\dd y\\
	=	\frac{ |\log(h_k)|}{ |\log(h_\ell)|} \(\frac{h_k}{h}\)^d \(\frac{h_k}{h_\ell}\)^{d/2-s} 
	\int \iota\(\frac{h_k}{h}y\) \alpha_\ell\(\frac{h_k}{h_\ell}(x-y)\)\dd y.
\end{multline*}
When $h=h_k/100$, then $h_k/ h$  is independent of $k$, so that
\begin{equation}\label{eq:bubble_rescaled}
 \mathcal R_k \(  \iota_{h_k/100}\ast \varphi_\ell \)= 
  \jmath\ast \varphi_{\ell,k}.
\end{equation}
The same argument works in the case $\ell=0$.
\end{proof}

\subsection{Norms of the bubbles}
\label{sec:bubble-norm}
We now estimate the Sobolev norms of the initial data.

\color{black}
\begin{lemma}[Sobolev norms of the initial bubbles]
\label{lem:norme-bulle-init}
Let $s'\geq 0$. When $\ell<k$, we have
\[
  \left\|\jmath\ast  \varphi_{\ell,k}\right\|_{\dot
    H^{s'}}\lesssim
  \frac{ |\log(h_k)|}{ |\log(h_\ell)|}
  \(\frac{h_k}{h_\ell}\)^{s'-s},
\]
whereas when $\ell>k$, there holds
\[
  \left\|\jmath\ast  \varphi_{\ell,k}\right\|_{\dot H^{s'}}\lesssim  
  	\frac{ |\log(h_k)|}{ |\log(h_\ell)|}  \(\frac{h_\ell}{h_k}\)^{s}.
\]
\end{lemma}
\color{black}
As a consequence, when $\ell<k$, then $\left\|\jmath\ast  \varphi_{\ell,k}\right\|_{\dot    H^{s'}}$ is large when $s'<s$ but small when $s'>s$, and when $\ell>k$, then $\left\|\jmath\ast  \varphi_{\ell,k}\right\|_{\dot   H^{s'}}$  is small for every $s'$.

\begin{proof}
When $k>\ell$, we have $h_k/h_\ell\to 0$, so  the initial data 
$\jmath\ast \varphi_{\ell,k}$ spread as $k\to \infty$. When $k<\ell$, the initial data $\jmath\ast \varphi_{\ell,k}$ rather concentrate, but the convolution  prevents the growth of Sobolev norms.

More precisely, let $s'\geq 0$, then
\begin{equation*}
  \| \varphi_{\ell,k}\|_{\dot H^{s'}}\lesssim  
\frac{ |\log(h_k)|}{ |\log(h_\ell)|}\(\frac{h_k}{h_\ell}\)^{s'-s} . 
\end{equation*}
Note that in $L^\infty$, we get a small norm when $\ell<k$ and  $s<s_c<d/2$:
\begin{equation*}
  \| \varphi_{\ell,k}\|_{L^\infty}\lesssim 
\frac{ |\log(h_k)|}{ |\log(h_\ell)|}\(\frac{h_k}{h_\ell}\)^{d/2-s} .
\end{equation*}

We now remark that the $L^1$-norms of $\jmath$ and its derivatives are
bounded independently of $k$.  
Using Young inequality, we deduce that  if $0\le \beta\le s'$, then
\begin{equation}\label{eq:estCIk}
  \left\|\jmath\ast  \varphi_{\ell,k}\right\|_{\dot H^{s'}}\le
  \|\jmath\|_{W^{1,\beta}} \|
  \varphi_{\ell,k}\|_{\dot H^{s'-\beta}}\lesssim 
   \frac{ |\log(h_k)|}{ |\log(h_\ell)|} \(\frac{h_k}{h_\ell}\)^{s'-\beta-s}. 
\end{equation}
As a consequence,
 when $\ell<k$, choosing $\beta=0$ we get the first inequality of the statement,
and when $\ell>k$, we choose $\beta=s'$ so that we get the second inequality of the statement.
\end{proof}

From these estimates, one can deduce upper bounds on the Sobolev norms
of $u_{0,k}$:
\color{black}
\[
\|u_{0,k}\|_{\dot H^{s'}}
	\lesssim \sum_{\ell<k}\frac{ |\log(h_k)|}{ |\log(h_\ell)|}\left(\frac{h_k}{h_\ell}\right)^{s'-s}+\sum_{\ell\geq k}\frac{ |\log(h_k)|}{ |\log(h_\ell)|}\left(\frac{h_\ell}{h_k}\right)^{s}.
\]
In the case $s'=s$, the sum over bubbles $\ell<k$ is convergent, but a
logarithmic factor remains,
$\sum_{\ell<k}\frac{ |\log(h_k)|}{ |\log(h_\ell)|}\lesssim
|\log(h_k)|$. Therefore, we get the upper bounds  
\[
\|u_{0,k}\|_{\dot H^{s'}}
	\lesssim
\begin{cases}
\(\frac{h_k}{h_{k-1}}\)^{s'-s}+1 
&\quad\text{if}\quad  s'> s,\\
|\log(h_k)|& \quad \text{if}\quad s'=s,\\
h_k^{s'-s}+1
	 & \quad\text{if}\quad s'< s.
\end{cases}
\]
The logarithmic unboundedness in the case $s'=s$ is essentially
irrelevant for the rest of this paper, as  in case it appears, it is
always multiplied by a positive power of $h_k$.
\color{black}
For $s'=0$, using the conservation of mass, we deduce the estimate
\begin{equation}\label{eq:L2}
\|u_k(t)\|_{L^2}\lesssim h_k^{-s}.
\end{equation}
As pointed out in the introduction, this estimate is in sharp contrast
with the case of  a single bubble considered in \cite{ACMA}: this
above bound, which is sharp, shows that the $L^2$-norm of the initial
data is not uniformly bounded in $k$. This forces us to adapt the
arguments from \cite{ACMA} at several stages: modulated energy
estimates, and interpolation steps to estimate $u_k$ in homogeneous
Sobolev spaces. 
\smallbreak

One can also estimate the semiclassical energy of $u_k$, which is
(formally) conserved by the flow of \eqref{eq:NLSsemi}:
\begin{equation*}
E_k(t)	=  \frac{\eps_k^2}{2}\|\nabla
  u_k(t)\|_{L^2}^2+\frac{1}{m+1}\|u_k(t)\|_{L^{2m+2}}^{2m+2}. 
\end{equation*}
Strictly speaking, for weak solutions
\color{black}{(a case we may have to consider, if $s_c>1$, as explained
  in Section~\ref{sec:def-set}),
 \color{black}
the energy is not necessarily 
conserved, but it is a nonincreasing
function of time; see \cite{GV85c}.
Let $0< s\leq s_c$. From the Sobolev embedding $\dot H^{s_{\rm sob}}\hookrightarrow L^{2m+2}$, we get 
\begin{equation*}
E_k(t)\le E_k(0)
	\lesssim \left(\varepsilon_k^2 \(\frac{h_k}{h_{k-1}}\)^{2(1-s)}+
        \(\frac{h_k}{h_{k-1}}\)^{2(m+1)(s_{\rm sob}-s)}\right)|\log(h_k)|+1. 
      \end{equation*}
In view of \eqref{eq:epsilon} and the algebraic relation
\begin{equation*}
 ms_c+1 = (m+1)s_{\rm sob},
\end{equation*}
the first term on the right hand side is controlled by the second one, and
 \begin{equation}\label{eq:H1}
E_k(t)	\lesssim \(\frac{h_k}{h_{k-1}}\)^{2(m+1)(s_{\rm sob}-s)}|\log(h_k)|+1.
\end{equation}
Note that $E_k$ stays bounded when $s< s_{\rm sob}$,
but tends to infinity  otherwise.

\begin{remark}[WKB analysis]\label{rem:WKB}
  We emphasize that due to the unboundedness, in
  $L^2(\R^d)$, of $u_{0,k}$, and
  possibly also
  $\eps_k\nabla u_{0,k}$ (if $s>s_{\rm sob}$), WKB analysis for
  \eqref{eq:NLSsemi} is not obvious at 
  all. Typically in the cubic case $m=1$, where the beautiful idea of
  Grenier \cite{Grenier98} makes it possible to justify WKB analysis
  in Sobolev spaces, the limit $k\to \infty$ in \eqref{eq:NLSsemi} is
  unclear. This is so even in Zhidkov spaces $X^s(\R^d)$ ($f\in X^s$, $s>1$,
  if $f\in L^\infty$ and $\nabla f\in 
  H^{s-1}$),  where WKB is justified
  in \cite{ACIHP}. We have seen above that $\eps_k\nabla u_{0,k}$ need not be
  bounded uniformly in $k$. We will bypass this difficulty by considering suitable
  modulated energy functionals, in Sections~\ref{sec:modulated} and
  \ref{sec:renormalized}. 
\end{remark}

\section{Analysis of semiclassical bubbles}
\label{sec:bubble}

In view of Lemma~\ref{lem:rescalingCI}, and of the semiclassical
analysis from \cite{ACMA}, introduce the hydrodynamical system
associated to the initial mode $\ell$, at scale $k$: 
\begin{equation}
  \label{eq:bulle}
  \left\{
    \begin{aligned}
      &\d_t \phi_{\ell,k}+\frac{1}{2}|\nabla \phi_{\ell,k}|^2+|a_{\ell,k}|^{2m}=0,\quad
      \phi_{\ell,k\mid t=0}=0,\\
      &\d_t a_{\ell,k} +\nabla \phi_{\ell,k}\cdot \nabla a_{\ell,k}+\frac{1}{2}a_{\ell,k}\Delta
      \phi_{\ell,k} =0,\quad a_{\ell,k\mid t=0}= \jmath \ast
      \varphi_{\ell,k}.
    \end{aligned}
  \right.
\end{equation}
We use the convention to denote $(\phi_{k,k},  a_{k,k})=(\phi_k,a_k)$. 
\subsection{Cauchy problem and zero speed of propagation}

Discarding the dependence upon the parameters $\ell$ and $k$, the set
of equations \eqref{eq:bulle} can be written in a universal way,
\begin{equation}
  \label{eq:bulle-univ-reg}
  \left\{
    \begin{aligned}
      &\d_t \phi+\frac{1}{2}|\nabla \phi|^2+|a|^{2m}=0,\quad
      \phi_{\mid t=0}=0,\\
      &\d_t a +\nabla \phi\cdot \nabla a+\frac{1}{2}a\Delta
      \phi =0,\quad a_{\mid t=0}= a_{\rm init},
    \end{aligned}
  \right.
\end{equation}
with the common feature that $a_{\rm init} \in
\mathcal{C}_{c}^{\infty}(\R^d) $. The following result will be of
constant use in the rest of this paper: 
\begin{lemma}\label{lem:MUK}
  Let $a_{\rm init} \in \mathcal{C}_{c}^{\infty}(\R^d) $. There exists
  $T>0$ and a unique solution $(\phi,a)\in \mathcal C([0,T];H^\infty(\R^d))$ to
  \eqref{eq:bulle-univ-reg}. Moreover, $(\phi,a)$ remains compactly
  supported for $t\in [0,T]$, and 
  \begin{equation*}
       \supp \phi(t,\cdot),\ \supp a(t,\cdot)\subset \supp a_{\rm init}.
  \end{equation*}
\end{lemma}
\begin{proof}
  This result is a consequence of the analysis from \cite{MUK86},
  whose main ideas we recall. Change of unknown function $(\phi,a)$
  to
  \begin{equation*}
    (V,A)=(\nabla \phi,a^m).
  \end{equation*}
  It solves
\begin{equation}
  \label{eq:bulle-univ-sym}
  \left\{
    \begin{aligned}
      &\d_t V+V\cdot \nabla V+\nabla\(|A|^{2}\)=0,\quad
      V_{\mid t=0}=0,\\
      &\d_t A +V\cdot \nabla A+\frac{m}{2}A\diver
      V =0,\quad A_{\mid t=0}=\(a_{\rm init}\)^m,
    \end{aligned}
  \right.
\end{equation}
which turns out to be a symmetric hyperbolic system, with a constant
symmetrizer. Indeed, denote
$U=(\RE(A), \IM(A), V)^T$: the system~\eqref{eq:bulle-univ-sym}
becomes 
\begin{equation}\label{eq:hyperbolic}
\partial_t U+\sum_{j=1}^dM_j(U)\partial_{j}U=0,
\end{equation}
where
\[
\sum_{j=1}^dM_j(U)\xi_j
	=\left(
	\begin{array}{ccc}
	V\cdot \xi & 0 & \frac m2\RE(A)\xi^T\\
	0 & V\cdot \xi & \frac m2 \IM(A)\xi^T\\
	2\RE(A)\xi & 2\IM(A) \xi & V\cdot \xi\operatorname{I}_d 
	\end{array}
	\right).
\]
Hence the matrices $M_j\in \mathcal{M}_{d+2}(\R)$ are such that $SM_j\in\mathcal{S}_{d+2}(\R)$, with
\[
S=\left(\begin{array}{cc}
m \operatorname{I}_d & 0\\
0 & 4\operatorname{I}_2
\end{array}\right).
\]
Local existence in $H^\si(\R^d)$ with $\si>d/2+1$ is then standard
for \eqref{eq:hyperbolic} hence for \eqref{eq:bulle-univ-sym}; see
e.g. \cite{Majda}. We emphasize tame estimates, which show that the
lifespan $T$ is independent of $\si>d/2+1$. 
Let $\Lambda =1-\Delta$: by symmetry,
  \begin{equation*}
    \frac{\dd}{\dd t}\<S\Lambda^\si U,\Lambda^\si U\>=2\<S\Lambda^\si
    \d_tU,\Lambda^\si U\> .
  \end{equation*}
In view of \eqref{eq:hyperbolic},
\begin{equation*}
  \<S\Lambda^\si
    \d_tU,\Lambda^\si U\> = -
    \sum_{j=1}^d \< S\Lambda^\si\(M_j(U)\d_j U\), \Lambda^\si U\> .
\end{equation*}
Write
\begin{align*}
  \< S\Lambda^\si\(M_j(U)\d_j U\), \Lambda^\si U\>  &= \<
    SM_j(U)\d_j \Lambda^\si U, \Lambda^\si U\>   \\
&\quad - \<
    S\left[M_j(U),\Lambda^\si\right]\d_j U, \Lambda^\si U\>.
\end{align*}
Since $SM_j$ is symmetric,
\begin{align*}
  \<  SM_j(U)\d_j \Lambda^\si U, \Lambda^\si U\>  =
-\frac{1}{2}\< \Lambda^\si U,\d_j\(  SM_j(U)\)\Lambda^\si U\>,
\end{align*}
As $SM_j$ is linear in its argument, we readily infer
\begin{equation*}
  \left| \<  SM_j(U)\d_j \Lambda^\si U, \Lambda^\si U\>\right| \lesssim \|
  S\d_j M_j(U)\|_{L^\infty} \|U\|_{H^\si}^2\lesssim \|\nabla
  U\|_{L^\infty}\|U\|_{H^\si}^2. 
\end{equation*}
By commutator estimate  (see~\cite{KaPo88}), we have
\begin{equation*}
  \left|\<
    S\left[M_j(U),\Lambda^\si\right]\d_j U,
    \Lambda^\si U\>\right|\lesssim \|\nabla U\|_{L^{\infty}}\|U\|_{H^\si}^2.
\end{equation*}
We infer:
\begin{equation}\label{eq:tame1}
    \frac{\dd}{\dd t}\<S\Lambda^\si U,\Lambda^\si  U\>\lesssim
   \|\nabla U\|_{L^{\infty}} \<S\Lambda^\si U,\Lambda^\si U\>. 
  \end{equation}
To return to the initial unknown $(\phi,a)$, we integrate the first
equation from \eqref{eq:bulle-univ-reg} with respect to time
\begin{equation*}
  \phi(t) = -\frac{1}{2}\int_0^t|V(\tau)|^2\dd\tau -
  \int_0^t|A(\tau)|^2\dd\tau. 
\end{equation*}
We check that $V=\nabla \phi$, as
\begin{equation*}
  \d_t \( \nabla \phi- V\) = \nabla \d_t \phi - \d_t V=0.
\end{equation*}
We go back to $a$ by now viewing the second equation in
\eqref{eq:bulle-univ-reg} as a transport equation with smooth
coefficients. This leads to the local existence result of the lemma.
  
We now move to the zero propagation speed property, established
initially in \cite[Theorem 2]{MUK86}. We see that there exists $C\ge
0$ such that for every $j$, 
\[
|M_j(U)|\leq C|U|.
\]
Fix $(t,x)$ in $[0,T]\times \R^d$. Since the solution is well-defined
on $[0,t]$, we can define 
\[
C'(t,x):=\sup_{t'\in[0,t], 1\leq j\leq d}|\partial_j U_k(t',x)|<+\infty.
\]
Using the equation, for $t'\in[0,t]$ there holds
\[
|\partial_tU_k(t',x)|
	\le d CC'(t,x)|U_k(t',x)|.
\]
Using Gronwall lemma, we deduce that $U_k(t,x)=0$ if and only if
$U_k(0,x)=0$, hence
\begin{equation*}
  \supp (A(t),V(t))=\supp (A,V)_{\mid t=0}=\supp a_{\rm init}.
\end{equation*}
We go back to the unknown $(\phi,a)$ like described above.
\end{proof}

\begin{remark}[Lifespan of the bubbles without convolution]
  Introduce bubbles without the initial convolution,
\begin{equation}
  \label{eq:bulle-app}
  \left\{
    \begin{aligned}
      &\d_t \uphi_{\ell,k}+\frac{1}{2}|\nabla \uphi_{\ell,k}|^2+|\ua_{\ell,k}|^{2m}=0,\quad
      \uphi_{\ell,k\mid t=0}=0,\\
      &\d_t \ua_{\ell,k} +\nabla \uphi_{\ell,k}\cdot \nabla \ua_{\ell,k}+\frac{1}{2}\ua_{\ell,k}\Delta
      \uphi_{\ell,k} =0,\quad \ua_{\ell,k\mid t=0}= 
      \varphi_{\ell,k}.
    \end{aligned}
  \right.
\end{equation}
Using the convention
$(\uphi_{\ell,\ell},\ua_{\ell,\ell})=(\uphi_\ell,\ua_\ell)$, we check
the algebraic relations
\color{black}
\begin{equation}\label{eq:scaling-bulle}
  \begin{aligned}
 \uphi_{\ell,k}(t,x) &=      \frac{\eps_k}{\eps_\ell}\uphi_\ell\(
                       \(\frac{\eps_kh_k^2}{\eps_\ell h_\ell^2}\)t, \frac{h_k}{h_\ell}x \)  ,\\
\ua_{\ell,k}(t,x)& = \frac{|\log   h_k|}{|\log h_\ell|}
\(\frac{h_k}{h_\ell}\)^{d/2-s}  \ua_\ell\(\(\frac{\eps_kh_k^2}{\eps_\ell
                 h_\ell^2}\) t, \frac{h_k}{h_\ell}x\) ,
   \end{aligned}
 \end{equation}
 \color{black}
 where we recall notation
 \color{black}
\begin{equation*}
  \eps_k = h_k^{m(s_c-s)}|\log h_k|^{m}.
\end{equation*}
\color{black}
  Using a virial computation, showing that if a global smooth solution
  exists for \eqref{eq:bulle-app}, then it is dispersive, the authors
  in \cite{MUK86} show that $T$, in Lemma~\ref{lem:MUK}, is
  necessarily finite. In view of \eqref{eq:scaling-bulle}, this implies
  that for $\ell<k$, $(\uphi_{\ell,k},\ua_{\ell,k})$ remains smooth on
  $[0,T_{\ell,k}]$, with
  \begin{equation*}
    T_{\ell,k} = \frac{\eps_\ell h_\ell^2}{\eps_k h_k^2} T\Tend k \infty +\infty,
  \end{equation*}
  while for $\ell>k$, $(\uphi_{\ell,k},\ua_{\ell,k})$ remains smooth on
  $[0,T_{\ell,k}]$ for the same expression of $T_{\ell,k}$, which now
  goes to zero as $k$ goes to infinity. 
\end{remark}
\begin{remark}[On low modes]\label{rem:low-modes}
  The scaling
  \eqref{eq:scaling-bulle} also shows that for $\ell<k$ and $t\approx
  1$, nonlinear effects are negligible in \eqref{eq:bulle-app}, a
  remark reminiscent of the strategy adopted in the proof of
  \cite[Proposition~2.6]{CaGa23}, where the \emph{linear} evolution of
  ``low modes'' is considered.  Actually, in this r\'egime, even the
  linear evolution of the initial data is negligible: the low modes
  are essentially constant in time, a remark which is exploited in the
  introduction of a renormalized modulated energy functional in
  Section~\ref{sec:renormalized}. 
\end{remark}

\subsection{Superposition principle}

Let us assume that the points $x_{\ell}$ are chosen so that the
profiles $\jmath\ast\varphi_{\ell,k}$ have disjoint supports. We
recall that $\alpha_k(x)=\alpha(x-\frac{x_k}{h_k})$. Assuming that
$\alpha$ and $\jmath=\iota(\cdot/100)$ are supported in
$\overline{B}(0,r_1)$ for some small constant $r_1>0$, we have  
\[
\operatorname{supp}(\jmath\ast\varphi_{\ell,k})
	\subset \overline{B}\(\frac{x_\ell}{h_k},2r_1\frac{h_\ell}{h_k}\).
\]
Choosing the points $x_{\ell}$ sufficiently far away from each other so that 
\[
|x_\ell-x_{\ell+1}|>4r_1h_\ell,
\]
the bubbles are therefore disjoint.

In this case, we may use a nonlinear superposition principle.
Introduce the following intermediate approximate solution with initial data including all the scales $\ell\leq k$:
\begin{equation}
  \label{eq:appint}
  \left\{
    \begin{aligned}
      &\d_t \tilde\phi_k+\frac{1}{2}|\nabla \tilde\phi_k|^2+|\tilde a_k|^{2m}=0,\quad
      \tilde \phi_{k\mid t=0}=0,\\
      &\d_t \tilde a_k +\nabla \tilde \phi_k\cdot \nabla \tilde a_k
+\frac{1}{2}\tilde a_k\Delta
      \tilde \phi_k =0,\quad \tilde a_{k\mid t=0}=\sum_{\ell\le k} \jmath \ast
      \varphi_{\ell,k}. 
    \end{aligned}
  \right.
\end{equation}
Lemma~\ref{lem:MUK} implies that
\begin{equation*}
  (\tilde\phi_k,\tilde a_k)= \sum_{\ell\le k} (\phi_{\ell,k},a_{\ell,k}),
\end{equation*}
where each bubble is given by \eqref{eq:bulle}. In view of the above
analysis, we also introduce ($\tilde
V_k,\tilde A_k)=(\nabla \tilde \phi_k,\tilde a_k^m)$, which solves
\begin{equation}
  \label{eq:Appint}
  \left\{
    \begin{aligned}
      &\d_t \tilde V_k+\tilde V_k\cdot \nabla \tilde
      V_k+\nabla\(|\tilde A_k|^{2}\)=0,\quad
      \tilde V_{k\mid t=0}=0,\\
      &\d_t \tilde A_k +\tilde V_k\cdot \nabla \tilde A_k
+\frac{m}{2}\tilde A_k\diver
      \tilde V_k =0,\quad \tilde A_{k\mid t=0}=\(\sum_{\ell\le k} \jmath \ast
      \varphi_{\ell,k}\)^m. 
    \end{aligned}
  \right.
\end{equation}
Since the functions $\jmath \ast \varphi_{\ell,k}$ have pairwise
disjoint supports, we may also write 
\begin{equation*}
     \tilde A_{k\mid t=0}=\sum_{\ell\le k} \(\jmath \ast
      \varphi_{\ell,k}\)^m   .
\end{equation*}
Like above, the zero speed of propagation from Lemma~\ref{lem:MUK}
then implies the relation
\begin{equation*}
  (\tilde V_k,\tilde A_k)= \sum_{\ell\le k} (V_{\ell,k},A_{\ell,k}),
\end{equation*}
where each $(V_{\ell,k},A_{\ell,k})$ solves the same system as
$(\tilde V_k,\tilde A_k)$, with initial datum
\begin{equation*}
  (V_{\ell,k},A_{\ell,k})_{\mid t=0} = \(0, \(\jmath \ast
      \varphi_{\ell,k}\)^m   \). 
\end{equation*}

\subsection{Refined estimates for the bubbles}

Like Lemma~\ref{lem:MUK}, the following lemma will be crucial for
the rest of this paper:
\begin{lemma}[Uniform estimates for the bubbles]\label{lem:tame}
There exist $C>0$ and $T>0$ such that  the smooth solutions $(\tilde
V_k,\tilde A_k)$ have a lifespan $T_k$ which is uniformly bounded from
below: $T_k\geq T>0$.  
Moreover, for every $t\in[0,T]$, for every $k\geq 1$ and for every
integer $\si$, there holds
\begin{align*}
\||\nabla |^\si \tilde V_k(t)\|_{L^2}+\||\nabla |^\si \tilde
  A_k(t)\|_{L^2}&\le C\||\nabla |^\si \tilde A_k(0)\|_{L^2}\\
 & \lesssim 1+ \(\frac{h_k}{h_{k-1}}\)^{\si +1 +m(s_c-s) -\frac{d}{2}}.
\end{align*}
\end{lemma}

Note that in particular, the Sobolev norms are bounded as soon as $\si+1\geq \frac d2$. Considering the equation satisfied by $(\tilde V_k,\tilde A_k)$, this implies that
\begin{equation*}
\| |\nabla |^\si\tilde V_k(t)\|_{L^\infty}+\||\nabla |^\si \tilde
  A_k(t)\|_{L^\infty}
  \lesssim \| |\nabla |^\si\tilde A_k(0)\|_{L^\infty}\\
  \lesssim 1+ \(\frac{h_k}{h_{k-1}}\)^{\si +1 +m(s_c-s)}.
\end{equation*}

\begin{proof}
  The proof relies on the symmetry of the hyperbolic system
  \eqref{eq:Appint}.
  Let $N>1+d/2$ and $\si\le N$ be integers,
 and $D^\si$ denote the family of differential operators in space, of
 order $\si$. In particular, the notation $\|D^\si U\|_{L^2}$ stands
 for
 \begin{equation*}
 \|D^\si U\|_{L^2}
 	=  \sum_{|\beta|=\si} \|\d_x^\beta U\|_{L^2}.
 \end{equation*}
 Like in the  framework recalled in the  proof of
 Lemma~\ref{lem:MUK}, denote
 \[
   U_k=(\RE(\tilde A_k), \IM(\tilde  A_k), \tilde V_k)^T.
 \]
 It solves \eqref{eq:hyperbolic}. First, we examine its initial data.
In view of
Section~\ref{sec:bubble-norm}, since $\si$ is an integer, and the
initial bubbles have pairwise disjoint supports, Leibniz rule and
H\"older inequality yield
\begin{align*}
  \| D^\si U_k (0)\|_{L^2}&=\sum_{\ell\le k}\| D^\si (\jmath\ast
   \varphi_{\ell,k})^m\|_{L^2} 
   \lesssim
  1+    \(\frac{h_k}{h_{k-1}}\)^{\si-ms+(m-1)\frac{d}{2}}.
\end{align*}
We rewrite the last power as
\begin{equation*}
  \si-ms+(m-1)\frac{d}{2}=\si+1-\frac{d}{2}+m(s_c-s).
\end{equation*}
In particular, for $d\le 2$, $U_k(0)$ is uniformly bounded in
$H^{N}(\R^d)$ (we choose $\si=0$), and for $d\le 4$, $\nabla U_k(0)$ is uniformly bounded in
$H^{N-1}(\R^d)$ (we choose $\si=1$). We start by proving Lemma~\ref{lem:tame} in
the case $d\le 4$. For $1\le n\le d$, $\d_n U$ solves
\begin{equation*}
  \d_t \d_n U_k +\sum_{j=1}^d M_j\( U_k\) \d_j \d_n U_k  =- \sum_{j=1}^d
  M_j\(\d_n U_k\) \d_j U_k. 
\end{equation*}
We proceed
classically, like in the proof of Lemma~\ref{lem:MUK}, by viewing the
term on the right hand side as a perturbative (semilinear) term, and compute
\begin{align*}
  \frac{\dd}{\dd t}\<S\Lambda^{N-1}\d_n U_k, \Lambda^{N-1}\d_n U_k\>& = 
-\sum_{j=1}^d \< S\Lambda^{N-1}\(M_j(U_k)\d_j \d_n U_k \),
 \Lambda^{N-1}\d_n U_k\>\\
&\quad -\sum_{j=1}^d \< S\Lambda^{N-1}\( M_j\(\d_n U_k\) \d_j U_k\),
 \Lambda^{N-1}\d_n U_k\>. 
\end{align*}
Using the fact that $H^{N-1}(\R^d)$ is a Banach algebra, and since
$M_j$ is linear in its argument, the last term is controlled by
$\|\nabla U_k\|_{H^{N-1}}^3$. For the first term on the right hand
side, we proceed exactly like in the proof of Lemma~\ref{lem:MUK}, and
write
\begin{align*}
  &\< S\Lambda^{N-1}\(M_j(U_k)\d_j \d_n U_k \),
 \Lambda^{N-1}\d_n U_k\> \\
&\qquad\qquad= \< S\(M_j(U_k)\d_j \Lambda^{N-1}\d_n U_k \),
 \Lambda^{N-1}\d_n U_k\>\\
&\qquad\qquad\quad + \< S\([\Lambda^{N-1},M_j(U_k)]\d_j \d_n U_k \),
 \Lambda^{N-1}\d_n U_k\>.
\end{align*}
By symmetry, the first term on the right hand side is controlled by
\begin{equation*}
  \|\nabla U_k\|_{L^\infty} \|\nabla U_k\|_{H^{N-1}}^2,
\end{equation*}
and by Kato-Ponce commutator estimate, the last term is estimated
similarly:
\begin{align*}
   \frac{\dd}{\dd t}\<S\Lambda^{N-1}\d_n U_k, \Lambda^{N-1}\d_n
   U_k\>&\lesssim \|\nabla U_k\|_{L^\infty} \|\nabla U_k\|_{H^{N-1}}^2
   + \|\nabla U_k\|_{H^{N-1}}^3\\
&\lesssim \|\nabla U_k\|_{H^{N-1}}^3,
\end{align*}
where we have used the Sobolev embedding.
Summing over $n\in \{1,\dots,d\}$, and recalling the equivalence of
norms
\begin{equation*}
  \<S\Lambda^{N-1}\nabla  U_k, \Lambda^{N-1}\nabla
   U_k\>\sim \|\nabla U_k\|_{H^{N-1}}^2,
\end{equation*}
we infer that there exists $T>0$ independent of $k$ such that $(\nabla
U_k)_k$ is uniformly bounded in $L^\infty([0,T];H^{N-1})$, and
Lemma~\ref{lem:tame} follows in the case $d\le 4$. 
\smallbreak

When $d\ge 5$, we replace the above quantity $\|\nabla
U_k\|_{H^{N-1}}$ by $\|D^{\si_0} U_k\|_{H^{N-\si_0}}$ for
$\si_0>1$:
\begin{equation*}
  \|U_k\|_{\si_0,N}:=\|D^{\si_0}
  U_k\|_{H^{N-\si_0}}=\sum_{\si=\si_0}^N \|D^\si U_k\|_{L^2}.  
\end{equation*}
Three conditions are required:
\begin{itemize}
 \item The initial norm $\|U_k(0)\|_{\si_0,N}$ is bounded uniformly in $k$. 
\item The norm $\|\cdot\|_{\si_0,N}$ controls the $L^\infty$-norm of $\nabla U_k$,
  \begin{equation*}
    \|\nabla U_k\|_{L^\infty}\lesssim  \|U_k\|_{\si_0,N}.
  \end{equation*}
\item This norm controls each term in the analogue of the above energy
  estimates.
\end{itemize}
As we have seen above, the first condition is fulfilled for $\si_0\ge
\frac{d}{2}-1$, and we choose
\begin{equation}\label{eq:si_0}
  \si_0 = 
  \begin{cases}
    \frac{d}{2}-1&\text{ if } d \text{ is even},\\
\left[\frac{d}{2}\right]&\text{ if } d \text{ is odd}.
  \end{cases}
\end{equation}
The second condition is satisfied, in view of the following elementary
result, valid in any space dimension:
\begin{lemma}\label{lem:Linfty}
Let $d\geq 1$, $\si_0$ given by \eqref{eq:si_0}, and $K>\frac
d2$. There exists $C=C(d,K)$ such that for every $f\in
H^{\si}(\R^d)$, 
\[
\|f\|_{L^{\infty}}\leq C\|D^{\si_0} f\|_{H^{K-\si_0}}.
\]
\end{lemma}
\begin{proof}[Proof of Lemma~\ref{lem:Linfty}]
We use the inverse Fourier transform to infer
\[
\|f\|_{L^{\infty}}
	\lesssim \int_{\R^d}|\widehat{f}(\xi)|\dd \xi.
\]
Cauchy-Schwarz inequality yields
\[
\left(\int_{\R^d}|\widehat{f}(\xi)|\dd \xi\right)^2
	\lesssim \(\int_{\R^d}|\xi|^{2\si_0}\langle
        \xi\rangle^{2(K-\si_0)}|\widehat{f}(\xi)|^2\dd \xi
        \)\(\int_{\R^d}\frac{1}{|\xi|^{2\si_0}\langle
          \xi\rangle^{2(K-\si_0)}}\dd \xi\). 
\]
Since $2K>d$, the last  integral is convergent
at infinity, and  our definition of $\si_0$ makes it convergent near
zero, as we always have $d-1-2\si_0>-1$. 
\end{proof}
Resume the strategy presented for $d\le 4$: the equation satisfied by
$D^{\si_0}U_k$ is of the form
\begin{equation*}
  \d_t D^{\si_0}U_k + \sum_{j=1}^d M_j(U_k) \d_j D^{\si_0}U_k = -
  \sum_{j=1}^d[ D^{\si_0},M_j(U_k)] \d_j U_k.
\end{equation*}
When considering
\begin{equation*}
  \frac{\dd}{\dd t}\<S \Lambda^{N-\si_0}D^{\si_0}U_k
  ,\Lambda^{N-\si_0}D^{\si_0}U_k\>, 
\end{equation*}
the terms coming from $M_j(U_k) \d_j D^{\si_0}U_k $ (left hand side
above) are treated like before, by using symmetry and Kato-Ponce
commutator estimate, as $ D^{\si_0}U_k $ is estimated in the
inhomogeneous Sobolev space $H^{N-\si_0}$. To control the terms coming from the
right hand side, by Leibniz rule and  Cauchy-Schwarz inequality, we
have to estimate a combination of terms of the form
\begin{equation*}
  \| D^{\si_1}U_k D^{\si_2}U_k\|_{L^2} ,\quad \si_1,\si_2\ge 1, \
  \si_0+1\le \si_1+\si_2 \le N+1.
\end{equation*}
Let us fix such exponents $\si_1$, $\si_2$. We have that
\begin{equation*}
  \| D^{\si_1}U_k D^{\si_2}U_k\|_{L^2} \le
  \|D^{\si_1}U_k\|_{L^p}\|D^{\si_2}U_k\|_{L^q},\quad
  \frac{1}{p}+\frac{1}{q}=\frac{1}{2}. 
\end{equation*}
In view the Sobolev embedding $\dot
H^{d(1/2-1/p)}(\R^d)\hookrightarrow L^p(\R^d)$, we infer
\begin{equation*}
  \| D^{\si_1}U_k D^{\si_2}U_k\|_{L^2} 
  \le
 \|D^{\si_1}U_k\|_{\dot H^{d/q}}\|D^{\si_2}U_k\|_{\dot H^{d/p}}.
\end{equation*}
 By symmetry of the roles, we assume that $\si_1\geq \si_2$. In everything that follows we assume that $\frac dq$ is an integer. We get
 \begin{align*}
  \| D^{\si_1}U_k D^{\si_2}U_k\|_{L^2}
  & \le \|D^{\si_1+ d/q}U_k\|_{L^2}
    \|D^{\si_2+[d/p]}U_k\|_{\dot H^{d/p-[d/p]}}\\
  & \le \|D^{\si_1+ d/q}U_k\|_{L^2}
  \|D^{\si_2+[d/p]}U_k\|_{H^1}.
\end{align*}
From a bootstrap argument based
on the control of $\|\nabla
U_k(t)\|_{L^\infty}$ uniformly in $k$ and $t\in [0,T]$, it suffices to
check that we may find such $p$ and $q$, satisfying in addition
\begin{equation*}
  \si_0\le \si_1+ [d/q]\le N,\quad \si_0\le \si_2+ [d/p]\le N-1.
\end{equation*}

We proceed as follows.
\begin{itemize}
\item If $\si_1\geq \si_0$, we choose $q=\infty$ and $p=2$. We get that 
\[
\si_1+[d/q] = \si_1 \in[ \si_0,N],
\]
\[
\si_2+[d/p]\geq [d/2]\geq \si_0. 
\]
Moreover, since $2\si_2\leq \si_1+\si_2\leq N+1$, we have
\[
\si_2+[d/p]\leq \frac{N+1}{2}+\frac d2,
\]
which is bounded above by $N-1$ when $N$ is chosen large enough.

\item
 Otherwise, we have $\si_2\leq \si_1\leq \si_0-1$. If $d=1$ we have $\si_0=0$ hence this case does not occur. Let $q\in[2,\infty)$ such that $\frac dq$ is an integer and
\[
\si_1+\frac dq
	=\si_1+[d/q]=\si_0.
\]
In particular we check that $\si_1+[d/q]\in[\si_0,N]$. 
Then $\si_1+\si_2\geq \si_0+1$ hence
\[
\frac dp=\frac d2-\frac dq=\frac d2-\si_0+\si_1\geq \frac d2+1-\si_2.
\]
Therefore, we have
\[
\si_2+\frac dp\geq \frac d2+1.
\]
In particular,
\[
\si_2+[d/p]\geq \frac d2\geq  \si_0.
\]
Finally,
\[
\si_2+[d/p]\leq \si_1+\frac d2\leq \si_0-1+\frac{d}{2},
\]
which is bounded above by $N-1$ when $N$ is chosen large enough.
\end{itemize}
\end{proof}

\section{Modulated energy estimate}
\label{sec:modulated}

We now use the approximate bubbles analyzed in the previous section in
order to establish some information regarding the actual solution $u_k$
to~\eqref{eq:NLSsemi}. We emphasize that even in the case of a single
bubble, $\tilde a_k e^{i
  \tilde \phi_k/\eps}$ must not be expected to approximate $u_k$
in $L^2$, due to phase modulations; see~\cite{CaARMA}.  But even if
phase modulations are taken into account, the unboundedness of
$u_{0,k}$ in $L^2$ 
is an issue to justify WKB analysis, see Remark~\ref{rem:WKB}. We follow the
strategy of~\cite[Theorem 4.1]{ACMA}, and 
introduce a modulated energy functional 
\[
H_k(t)=\frac 12 \|(\eps_k\nabla -i\tilde V_k)u_k\|_{L^2}^2+\int_{\R^d}
  \(F(|u_k|^2)-F(\tilde \rho_k) -(|u_k|^2-\tilde \rho_k)f(\tilde
  \rho_k)\)\dd x,
\]
where $\tilde \rho_k=|\tilde a_k|^2$, $f(y)=y^m$, and
\[
F(y)=\int_0^y f(z)\dd z=\frac{y^{m+1}}{m+1}.
\]
Denote the kinetic part by
\[
K_k(t)=\frac 12 \|(\eps_k\nabla -i\tilde V_k)u_k\|_{L^2}^2
\]
and the potential part by
\[
P_k(t)
	=\int_{\R^d}
  \(F(|u_k|^2)-F(\tilde \rho_k) -(|u_k|^2-\tilde \rho_k)f(\tilde
  \rho_k)\)\dd x.
\]
We know from \cite{ACMA} that the potential part is bounded from below by
\[
P_k(t)
	\ge c\int_{\R^d}\left| |u_k|^2-\tilde\rho_k\right|(|u_k|^{2m}+(\tilde\rho_k)^{m})\dd x,
\]
for some $c>0$ depending only on $m$. In particular, $H_k$ is the sum
of two nonnegative terms. 

We first detail the case  $s<s_{\rm sob}$ (with $s\le 2$), for the sake of clarity.
When $t=0$,  we get, in view of Section~\ref{sec:bubble-norm},
\[
K_k(0)=\frac 12 \|\eps_k\nabla u_{0,k}\|_{L^2}^2
	\lesssim \eps_k^2h_k^{2(1-s)}|\log(h_k)|.
      \]
 Under the assumption   $s<s_{\rm sob}$,  $K_k(0)=\mathcal{O}(\varepsilon_k^2)$ as
 $k\to\infty$. Indeed, we note the algebraic relation
 \color{black}
\begin{equation*}
\eps_k^2h_k^{2(1-s)}|\log(h_k)|= |\log h_k|^{2m+1} h_k^{2m(s_c-s)+2(1-s)} =
|\log h_k|^{2m+1}  h_k^{(2m+2)(s_{\rm sob}-s)}.
\end{equation*}
\color{black}
As for the initial potential part, we first use the fact that $P_k$
corresponds to the beginning of a Taylor expansion,
\begin{equation*}
  P_k= \frac{1}{2}\int_{\R^d} \(|u_k|^2-\tilde \rho_k\)^2\int_0^1
  f'\( \tilde \rho_k + \theta(|u_k|^2-\tilde \rho_k)\)\dd \theta\dd x.
\end{equation*}
Considering this quantity at time $t=0$, and using the fact that the
supports of the initial bubbles  $\jmath\ast \varphi_{\ell,k}$ are
pairwise disjoint,
\begin{equation*}
  \supp \jmath\ast \varphi_{\ell,k}\cap \supp \jmath\ast
  \varphi_{\ell',k}=\emptyset,\quad \forall \ell\not =\ell',
\end{equation*}
we obtain the rough bound
\begin{equation*}
  0\le P_k(0)\lesssim  \sum_{\ell>k} \int_{\R^d} |\jmath\ast
  \varphi_{\ell,k}|^4 \times \sum_{\ell'} |\jmath\ast
  \varphi_{\ell',k}|^{2m-2}  \lesssim
  \sum_{\ell>k} \|\jmath\ast \varphi_{\ell,k}\|^{2m+2}_{L^{2m+2}}.
\end{equation*}
In view of the
embedding $\dot H^{s_{\rm sob}}\hookrightarrow L^{2m+2}$, we infer,
thanks to Lemma~\ref{lem:norme-bulle-init}, and since the map
$z\mapsto M^{z^2} e^{-M^z}$ is integrable on $[1,\infty)$ and
decreasing for $z\ge z_0\gg 1$,
\color{black}
\begin{equation*}
  P_k(0)\lesssim \(\frac{|\log (h_{k})|}{|\log
    (h_{k+1})|}\)^{2m+2} \(\frac{h_{k+1}}{h_k}\)^{(2m+2)s}|\log(h_k)|. 
\end{equation*}
\color{black}
Choosing $M>1$ sufficiently large, we have in particular
$P_k(0)=\O(\eps_k^2)$. 

\begin{proposition}[Modulated energy estimate]\label{prop:modulated0}
Let $s<s_{\rm sob}$ with $s\le 2$. Then for every $t\in [0,T]$,where $T$ is
given by Lemma~\ref{lem:tame}, there holds
\[
H_k(t)
	\lesssim 
	\eps_k^2+h_k^{2(m+1)(s_{\rm sob}-s)}.
\]
\end{proposition}

\begin{proof}
We estimate the time derivative of $H_k$.

We follow line by line the proof of~\cite[Theorem 4.1]{ACMA}, taking into account the fact that the $L^2$-norm of $u_k$ has the upper bound 
\[
\|u_k\|_{L^2}=\|u_{0,k}\|_{L^2}\lesssim h_k^{-s},
\] hence it is not necessarily bounded as $k\to\infty$. Similarly, low order Sobolev norms of $\tilde V_k$ may grow to infinity as $k\to\infty$.

Since we only consider weak solutions $u_k$, some integrations by part may not make sense. For simplicity, the following proof will hence be formal only, but can be fully justified by working on a sequence $u_k^{(n)}\in\mathcal{C}(\R, H^2)$ of global strong solutions to the equation
\[
i\eps_k\partial_t u_k^{(n)}+\frac{\eps_k^2}{2}\Delta
u_k^{(n)}=f_n\(\left|u_k^{(n)}\right|^2\) 
\quad ; \quad
u_{k\mid t=0}^{(n)}=u_{0,k},
\]
\[
f_n(y)=\frac{y^m}{1+(\delta_n y)^m},
\quad \delta_n\to 0^+.
\]
Indeed, as $n\to\infty$, the sequence $(u_k^{(n)})_n$ will converge to a weak solution to~\eqref{eq:NLSsemi} with initial data $u_{0,k}$. In this case, one should replace $F$ by $F_n(y)=\int_0^yf_n(z)\dd z$ in the formula for the modulated energy $H_k^{(n)}$, and $G(y)=\int_0^yzf'(z)\dd z$ by $G_n(y)=\int_0^yzf_n'(z)\dd z$.
\smallbreak

In order to estimate $H_k(t)$, we use Madelung transform and switch to the hydrodynamic variables 
\[\rho=|u_k|^2,
\quad
J=\IM\(\eps_k \overline{u_k}\nabla u_k\),
\]
satisfying the equation
\[
  \left\{
\begin{aligned}
&\partial_t\rho+\operatorname{div}(J)=0
\\
&\partial_tJ_j+\frac{\eps_k^2}{4}\sum_{p=1}^d \partial_i\(4\RE(\partial_j \overline{u_k}\partial_p u_k\)-\partial_{pj}\rho)+\partial_j G(\rho)=0.
\end{aligned}
\right.
\]
As in~\cite{ACMA}, we write the time derivative of the kinetic energy using the hydrodynamic variables, then proceed by integration by parts. We get the formula
\begin{equation*}
\frac{\dd}{\dd t}H_k
	=-\frac 14 I-I'-\int(G(\rho)-G(\tilde \rho_k)-(\rho-\tilde \rho_k)G'(\tilde \rho_k)\diver \tilde V_k,
      \end{equation*}
where
\[
I =\eps_k^2\int \nabla \(\diver \tilde V_k\)\cdot \nabla |u_k|^2,
\]
\[
I'=\RE\sum_{p,j}\int\partial_j  (\tilde V_k)_p (\eps_k\partial_j u_k-i (\tilde V_k)_j u_k)(\overline{\eps_k\partial_p u_k-i (\tilde V_k)_p u_k}).
\]
We note that $\|\nabla \tilde V_k\|_{L^{\infty}}$ is uniformly bounded, hence $I'=\O(K_k)$.

The only new difficulty, compared to \cite{ACMA},  is to show that
$I=\mathcal{O}(K_k+\eps_k^2)$. We
have 
  \begin{equation*}
 I 
	= \eps_k\int
     \nabla \(\diver \tilde V_k\)\cdot\( \overline u_k (\eps_k\nabla-i\tilde V_k)u_k
  +u_k \overline {(\eps_k\nabla-i\tilde V_k)u_k} \).
\end{equation*}
We note that if we proceed like in \cite{ACMA}, we invoke Cauchy-Schwarz then Young inequalities to get the upper bound
\begin{equation*}
\left|  I  \right|\\
\leq \|(\eps_k\nabla-i \tilde V_k)u_k\|_{L^2}^2 + \eps_k^2 \|u_k\|_{L^2}^2,
\end{equation*} 
which may go to infinity as $k\to\infty$ since we only know that
$
  \|u_{k}\|_{L^2}\lesssim h_k^{-s}.
$

Instead, we decompose
\[
|I|\leq  \eps_k\int
    \left| \nabla \left(\diver \tilde V_k\right)\right|\times \left| u_k\right| \times \left|(\eps_k\nabla-i\tilde V_k)u_k\right|
    	 = II+III
 \]
where
\[
 II=\eps_k\int
    \left| \nabla \left(\diver \tilde V_k\right)\right|\times \left(\left| u_k\right| - \sqrt{\tilde \rho_k} \right) \times \left|(\eps_k\nabla-i\tilde V_k)u_k\right|,
\]
\[
III =\eps_k\int
    \left| \nabla \left(\diver \tilde V_k\right)\right|\times \sqrt{ \tilde \rho_k} \times \left|(\eps_k\nabla-i\tilde V_k)u_k\right|.
\]
Then using H\"older inequality,
\[
| II |
	\leq\eps_k \left\|\nabla \diver \tilde V_k\right\|_{L^{2+\frac
            2m}}\left\||u_k|-\sqrt{\tilde \rho_k}
          \right\|_{L^{2m+2}}\|(\eps_k\nabla-i\tilde V_k)u_k\|_{L^2}. 
\]
We know thanks to Lemma~\ref{lem:tame} and Sobolev embedding, that for
$N$ sufficiently large,
\[
\left\|\nabla \diver \tilde V_k\right\|_{L^{2+\frac 2m}}\lesssim 
\left\|\nabla \diver \tilde V_k\right\|_{H^{N}}\lesssim 1+h_k^{3+m(s_c-s)-\frac{d}{2}}.
\]
Rewriting
\begin{equation*}
  3+m(s_c-s)-\frac{d}{2} = (m-1)\left(\frac{d}{2}-s\right)+2-s,
\end{equation*}
we see that this power is nonnegative as soon as $s\le 2$. 
Moreover, the last two factors are estimated respectively by
\begin{align*}
  \left\||u_k|-\sqrt{\tilde \rho_k}
          \right\|_{L^{2m+2}}^{2m+2}&=\int \left||u_k|-\sqrt{\tilde \rho_k}
                                      \right|^{2m+2}\\
  &\lesssim \int \left| |u_k|^2-\tilde
            \rho^k\right|^2\(|u_k|^{2m} + \(\tilde
          \rho_k\)^m\)\lesssim P_k,
\end{align*}
and
\begin{equation*}
  \|(\eps_k\nabla-i\tilde V_k)u_k\|_{L^2}^2=2K_k.
\end{equation*}
Therefore, 
\begin{equation*}
| II |\lesssim \varepsilon_kP_k^{\frac{1}{2m+2}}\sqrt{ K_k}
  \lesssim  K_k + \eps_k^2P_k^{\frac{1}{m+1}}\lesssim K_k + P_k +
    \eps_k^{2\frac{m+1}{m}}, 
\end{equation*}
where we have used Young inequality twice. 
Finally, using a similar strategy, we obtain the estimate
\[
| III |
\leq\eps_k \left\|\nabla \diver \tilde V_k\right\|_{L^{2}}
\left\|\sqrt{\tilde \rho_k} \right\|_{L^{\infty}}\|(\eps_k\nabla-i\tilde V_k)u_k\|_{L^2}.
\]
In view of Lemma~\ref{lem:MUK}, \eqref{eq:scaling-bulle},
Lemma~\ref{lem:tame}, and Sobolev embedding, we
 have
 \begin{equation*}
   \|\tilde \rho_k\|_{L^\infty([0,T];L^\infty)}\lesssim 1.
 \end{equation*}
To estimate $\nabla \diver \tilde V_k$, that is $\nabla^3 \tilde
\phi_k$, in $L^2$, we invoke Lemma~\ref{lem:tame} like above, to have,
since $s\le 2$,
\begin{equation*}
  \left\|\nabla \diver \tilde
    V_k\right\|_{L^\infty([0,T];L^{2})}\lesssim 1.
\end{equation*}
 Therefore,
 \[
| III |
	\lesssim 
        \eps_k \sqrt{ K_k}\lesssim \eps_k^2+K_k.
      \]
 We then conclude as in the proof of~\cite[Theorem 4.1]{ACMA} that
\[
\frac{\dd H_k}{\dd t}
	\lesssim H_k+\eps_k^2,
\]
and the proposition follows from a Gronwall argument.
\end{proof}

\section{Renormalized modulated energy estimate}
\label{sec:renormalized}

When  $s\geq s_{\rm sob}$, the approach from
Proposition~\ref{prop:modulated0} is not satisfying since the
modulated energy functional is large even at time $t=0$, due to the
kinetic part. It turns out
that this initial value is the main responsible for the failure of the
approach: we renormalize the modulated energy functional by removing
the initial bubbles $\ell<k$ (those which have large lower-order
Sobolev norms). More precisely, we introduce  
\[
\tilde H_k(t)=\tilde K_k(t)+P_k(t),
\]
where we denote the renormalized kinetic part by
\[
\tilde K_k(t)=\frac 12 \|(\eps_k\nabla -i\tilde V_k)u_k-\eps_k\nabla\varphi\|_{L^2}^2,
\quad \varphi=\sum_{\ell<k} \jmath \ast \varphi_{\ell,k},
\]
and we leave the potential part unchanged.
With this choice,  at $t=0$,
\[
\tilde H_k(0)
	\lesssim \varepsilon_k^2.
\]
Note that this strategy somehow meets the approach followed in the
proof of \cite[Proposition~2.6]{CaGa23}, as noted in
Remark~\ref{rem:low-modes}. 
\begin{proposition}[Renormalized modulated
  energy]\label{prop:modulated} Let $0<s<s_c$, with $s\le 2$. For
  every $t\in[0,T]$, there holds 
\[
\tilde H_k(t)\lesssim \eps_k^2.
\]
\end{proposition}

The rest of this section if devoted to the proof of Proposition~\ref{prop:modulated}.
To emphasize the new terms compared to Section~\ref{sec:modulated}, we
develop 
\[
\tilde K_k(t)=K_k(t)- L_k(t) +\|\eps_k\nabla\varphi\|_{L^2}^2,
\]
\[
L_k(t)
	=2\RE\langle (\eps_k\nabla -i\tilde V_k)u_k, \eps_k\nabla\varphi\rangle,
\]
then we evaluate the time derivative of the two components $K_k(t)$
and $L_k(t)$ separately ($\nabla \varphi$ is obviously time-independent).

\subsection{Localized Sobolev norms of the WKB ansatz}

Let $\chi_k\in C_c^\infty(\R^d)$ such that $\chi_k$ localizes in space around the $k$-th bubble $x\approx \frac{x_k}{h_k}$:
\begin{align*}
   \chi_k\equiv 1& \quad\text{on }\quad \operatorname{supp}  (\jmath\ast \varphi_{k,k})=\operatorname{supp}( \jmath\ast\alpha_k),\\
   \chi_k\equiv 0 & \quad\text{on }\quad \operatorname{supp}   (\jmath\ast \varphi_{\ell,k}) \quad  \text{for}\quad  \ell\neq k.
\end{align*}

On the support of $(1-\chi_k)$, we will use the fact that the Sobolev
norms of $(\tilde V_k,\tilde A_k)$ are proportional to the restriction
of the Sobolev norms of its initial data $(0,(\tilde \rho_k)^m)$ to
the support of $(1-\chi_k)$, hence the higher order Sobolev norms are
decaying in $k$.
More precisely we establish the following lemma.
\begin{lemma}[Localized Sobolev norms of the approximate phase]\label{lem:V-loc} For every $n\geq 0$, there holds
\[
\|(1-\chi_k)|\nabla|^n \tilde V_k \|_{L^{\infty}}
     	\lesssim \eps_k^2h_k^{3+n}.
\]
\end{lemma}

\begin{proof}
We note that at time $t=0$, we have $|\nabla|^n \tilde
V_{k\mid t=0}=0$. Hence we use the equation~\eqref{eq:Appint} satisfied by
$\tilde V_k$: 
\begin{equation*}
\frac{\dd}{\dd t}|\nabla|^n \tilde V_k
	=-|\nabla|^n\( \tilde V_k\cdot \nabla \tilde V_k- \nabla( |\tilde A_k|^2)\).
\end{equation*}
Then we note that on the support of $(1-\chi_k)$, the Sobolev norms of $|\tilde A_k|^2 $ are proportional to the Sobolev norms of the restriction of its initial data $(\tilde \rho_k)^{m}$ according to Lemma~\ref{lem:tame}. Therefore the equality right below Lemma~\ref{lem:tame} taking into account the localization implies the estimate
\[
\|(1-\chi_k)|\nabla|^n \nabla( |\tilde A_k|^2)\|_{L^{\infty}}
	\lesssim h_k^{1+n+2m(\frac d2-s)}=h_k^{3+n}\eps_k^2.
\]
Let us now treat the convective part. Similarly, this term vanishes at
time $t=0$ hence we compute one more time derivative: 
\begin{multline*}
\frac{\dd}{\dd t}|\nabla|^n \( \tilde V_k\cdot \nabla \tilde V_k\)
	=-|\nabla|^n\( \(\tilde V_k\cdot \nabla \tilde V_k- \nabla( |\tilde A_k|^2)\)\cdot\nabla\tilde V_k\)\\
	-|\nabla|^n\( \tilde V_k \cdot\nabla\(\tilde V_k\cdot \nabla \tilde V_k- \nabla( |\tilde A_k|^2)\)\).
\end{multline*}
The terms involving $\tilde A_k$ are small thanks to the same argument
as before. Regarding the terms involving only $\tilde V_k$, the terms are of the form 
\[
\||\nabla|^n( (\tilde V_k\cdot \nabla \tilde V_k)\cdot\nabla \tilde V_k)\|_{L^{\infty}}
	\lesssim h_k^{2+n}h_k^{3m(\frac d2-s)}
	\lesssim h_k^{5+n}\eps_k^{3},
      \]
hence the lemma after integration in time. 
\end{proof}

\subsection{Linear part}
We compute
\begin{align*}
\frac{\dd}{\dd t}L_k(t)
	=&2\RE\langle (\eps_k\nabla -i\tilde V_k)\partial_t u_k, \eps_k\nabla\varphi\rangle
	+2\RE\langle  -i \partial_t\tilde V_k u_k, \eps_k\nabla\varphi\rangle\\
	=&-\RE\langle i\varepsilon_k \Delta u_k,(\eps_k\nabla -i\tilde V_k)  \eps_k \nabla\varphi\rangle
	+2\RE\langle i|u_k|^{2m}u_k,  (\eps_k\nabla -i\tilde V_k)\cdot \nabla\varphi\rangle\\
	&-2\RE\langle  i \partial_t\tilde V_k u_k, \eps_k\nabla\varphi\rangle.
\end{align*}

In the subsequent estimates, we leave out the dependence of $\eps_k$
upon $\log(h_k)$ (see \eqref{eq:epsilon})  to lighten the notations,
since, as pointed out before, the logarithmic correction is irrelevant
when we deal with open algebraic conditions. 
\smallbreak

$\bullet$ First, we note that
\[
\left|\langle i\eps_k\Delta u_k,  \eps_k^2 \Delta \varphi\rangle\right|
	\leq \|\eps_k\nabla u_k\|_{L^2}\| \eps_k^2 \nabla\Delta\varphi\|_{L^2}
	\lesssim \eps_k^3 h_k^{4-2s},
 \]
 where we have used \eqref{eq:H1} to estimate $\eps_k\nabla u_k$, and
 the first part of 
 Lemma~\ref{lem:norme-bulle-init} with $s'=3$ to estimate
 $\nabla\Delta\varphi$.  
\smallbreak

$\bullet $ We now estimate $\langle  i|u_k|^{2m}u_k, \eps_k\Delta\varphi\rangle$. First, we decompose
\[
\langle  i|u_k|^{2m}u_k, \eps_k\Delta\varphi\rangle
	=\langle  i(|u_k|^{2m}-|\tilde a_k|^{2m})u_k, \eps_k\Delta\varphi\rangle
	+\langle  i|\tilde a_k|^{2m}u_k, \eps_k\Delta\varphi\rangle.
\]
The first term on the right hand side has the upper bound
\[
\left|\langle  i(|u_k|^{2m}-|\tilde a_k|^{2m})u_k,
  \eps_k\Delta\varphi\rangle \right|
	\lesssim \||u_k|^{2m}-|\tilde a_k|^{2m}\|_{L^{\frac{2m+2}{2m}}}\|u_k\|_{L^{2m+2}}\|\eps_k\Delta\varphi\|_{L^{2m+2}}.
\]
This implies, in view of \eqref{eq:H1}, the embedding $\dot H^{s_{\rm
    sob}}\hookrightarrow L^{2m+2}$, Lemma~\ref{lem:norme-bulle-init},
and Young inequality, 
\begin{align*}
\left|\langle  i(|u_k|^{2m}-|\tilde a_k|^{2m})u_k, \eps_k\Delta\varphi\rangle\right|
	&\lesssim P_k^{\frac{2m}{2m+2}} h_k^{\frac d2-s-\frac{d}{2m+2}} \eps_kh_k^{2+\frac d2-s-\frac{d}{2m+2}}\\
	&\lesssim P_k+ \eps_k^{m+1}h_k^{2(m+1)+2(m+1)(\frac d2-s)-d}\\	
	&\lesssim P_k+ \eps_k^{m+1}h_k^{2(m+1)+2m(\frac d2-s)-2s}.
\end{align*}
We note that $2m(\frac d2-s)-2s=2+2m(s_c-s)$, so $h_k^{2m(\frac
  d2-s)-2s} = \eps_k^2h_k^2$. 
Let us now investigate the second term $\langle  i|\tilde a_k|^{2m}u_k, \eps_k\Delta\varphi\rangle$:
\[
\left|\langle  i|\tilde a_k|^{2m}u_k, \eps_k\Delta\varphi\rangle\right|
	\lesssim \eps_k\|\tilde a_k\|_{L^{\infty}}^{2m}\|u_k\|_{L^2}\|\Delta\varphi\|_{L^2}
	\lesssim \eps_k h_k^{2m\left(\frac d2-s\right)+2-2s}=\eps_k^3 h_k^{4-2s}.
\]
We conclude that when $t\lesssim 1$,
\[
\left|\langle  i|u_k|^{2m}u_k, \eps_k\Delta\varphi\rangle\right|
	\lesssim P_k
	+\eps_k^3 h_k^{4-2s}.
\]

$\bullet$ Then we note that
\[
\left|\langle \varepsilon_k \Delta u_k,\tilde V_k \cdot \eps_k \nabla\varphi\rangle\right|
	= \eps_k^2\left|\langle  u_k,\Delta(\tilde V_k \cdot  \nabla\varphi)\rangle\right|.
\]
As a consequence, since $\nabla \varphi = (1-\chi_k)\nabla\varphi$, using the conservation of $L^2$ norm and 
Lemma~\ref{lem:V-loc},

\[
\left|\langle \varepsilon_k \Delta u_k,\tilde V_k \cdot \eps_k \nabla\varphi\rangle\right|
	\lesssim \eps_k^2 h_k^{3-2s}\eps_k^2 h_k^{3}.
\]

$\bullet$  Moreover, \eqref{eq:H1} and Lemma~\ref{lem:V-loc} yield
\begin{multline*}
\left|\langle i|u_k|^{2m}u_k, i\tilde V_k\cdot \nabla\varphi\rangle\right|
	\lesssim \|u_k\|_{L^{2m+2}}^{2m+1}\|(1-\chi_k)\tilde V_k\|_{L^{\infty}}\|\nabla \varphi\|_{L^{2m+2}}\\
	\lesssim \eps_k^2 h_k^{2-2s+1} \eps_k^2h_k^3.
\end{multline*}

$\bullet$ Finally, we estimate $\langle  i \partial_t\tilde V_k u_k, \eps_k\nabla\varphi\rangle$ on the support of $\varphi$. 

Using the equation~\eqref{eq:Appint} satisfied by $\tilde V_k$ and
Lemma~\ref{lem:V-loc}, we have the estimate
\[
\|\partial_t \tilde V_k\|_{L^{\infty}}
	\lesssim \eps_k^4 h_k^7,
\]
and so 
\[
\left|\langle  i \partial_t\tilde V_k u_k, \eps_k\nabla\varphi\rangle\right|
	\lesssim \eps_k^5 h_k^{8-2s}.
\]

$\bullet$ In conclusion,
\[
\left| \frac{\dd}{\dd t}L_k(t)\right|
	\lesssim P_k(t)+ \eps_k^3 h_k^{4-2s}.
\]

\subsection{Nonlinear part}

We now also replace $K_k$ by $\tilde K_k$ in the upper bounds for
estimating $\frac{\dd K_k}{\dd t}$ that appear in the proof
of~\cite[Theorem~4.1]{ACMA}.

$\bullet$ First, we treat the following problematic term evidenced in the proof of Proposition~\ref{prop:modulated0},
 \begin{equation*}
 I = \eps_k^2\int
     \nabla \(\diver \tilde V_k\)\cdot \nabla (|u_k|^2).
\end{equation*}
The same argument as in Section~\ref{sec:modulated} would lead to the inequalities
\begin{align*}
 &\left| \eps_k\int
     \nabla \(\diver \tilde V_k\)\cdot\( \overline u_k (\eps_k\nabla u_k-i\tilde V_k u_k-\eps_k\nabla \varphi \right) \right|\\
&\qquad\qquad      \lesssim \eps_k
      (1+h_k^{2-s})\(P_k^{\frac{1}{2m+2}}+1\)\sqrt{\tilde K_k}, 
\end{align*}
\begin{equation*}
 \left| \eps_k^2\int
     \nabla \(\diver \tilde V_k\)\cdot\( \overline u_k \nabla \varphi\) \right|
      \lesssim \eps_k^2 h_k^{3-2s}\(P_k^{\frac{1}{2m+2}}+1\).
\end{equation*}
However, since we want to tackle exponents $s$ that may be larger than
$\frac 32$, we refine slightly the argument in order to improve the
estimate.  Indeed, as we wish $s$ to be arbitrarily close to $s_c$,
the power of $\eps_k$ is essentially useless, and we need the
remaining power of $h_k$ to be nonnegative. 

Lemma~\ref{lem:V-loc} taken with two derivatives leads to
\[
\|(1-\chi_k)\nabla \diver \tilde V_k \|_{L^{\infty}}
     	\lesssim \eps_k^2h_k^{5},
\]
and applying this estimate in the formula for
  \begin{equation*}
 I 
	= \eps_k\int
     \nabla \(\diver \tilde V_k\)\cdot\( \overline u_k (\eps_k\nabla-i\tilde V_k)u_k
  +u_k \overline {(\eps_k\nabla-i\tilde V_k)u_k} \).
\end{equation*}
we get
\[
| I |\lesssim \tilde H_k + \eps_k^2 h_k^{4-2s}.
\]

$\bullet$ We also need to estimate the second problematic term from the proof of Proposition~\ref{prop:modulated}
\[
I'
	:=-\RE\sum_{i,j}\int\partial_j  (\tilde V_k)_i (\eps_k\partial_j u_k-i (\tilde V_k)_j u_k)(\overline{\eps_k\partial_i u_k-i (\tilde V_k)_i u_k}).
\]
We decompose again $(\eps_k\nabla-i\tilde V_k)u_k=(\eps_k\nabla-i\tilde V_k)u_k-\eps_k\nabla \varphi +\eps_k\nabla\varphi$ to get
\begin{multline*}
| I' |
	\lesssim \tilde K_k 
	+ 2\sum_{i,j}\int|\partial_j  (\tilde V_k)_i| |\eps_k\partial_j\varphi| |\eps_k\partial_i u_k-i (\tilde V_k)_i u_k|\\
	+\sum_{i,j}\int|\partial_j  (\tilde V_k)_i| |\eps_k\partial_j \varphi| |\eps_k\partial_i \varphi|.
\end{multline*}
The second term on the right hand side can be estimated from the first and third terms on the right hand side.
Finally, Lemma~\ref{lem:V-loc}  implies 
\[
\|(1-\chi_k)\partial_j \tilde V_k \|_{L^{\infty}}
     	\lesssim \eps_k^2h_k^{4},
\]
so that ,since $(1-\chi_k)=1$ on the support of $\varphi$,
\[
\int|\partial_j  (\tilde V_k)_i| |\eps_k\partial_j \varphi| |\eps_k\partial_i \varphi|
	\lesssim \eps_k^2 h_k^{6-2s}.
\]
This leads to
\[
| I' |
	\lesssim \tilde K_k + \eps_k^2 h_k^{6-2s}.
\]

$\bullet$ In conclusion, we get 
\[
\left|\frac{\dd \tilde H_k}{\dd t}\right|
	\lesssim \tilde H_k + \eps_k^2h_k^{4-2s},
\]
hence Proposition~\ref{prop:modulated} by integration.

\section{Localized estimates and conclusion}
\label{sec:local}

\subsection{Localized mass estimate}

Like in Section~\ref{sec:renormalized}, let $\chi_k\in
C_c^\infty(\R^d)$ such that $\chi_k$ localizes in space around the
$k$-th bubble $x\approx \frac{x_k}{h_k}$: 
\begin{align*}
   \chi_k\equiv 1& \quad\text{on }\quad \operatorname{supp}  (\jmath\ast \varphi_{k,k})=\operatorname{supp}( \jmath\ast\alpha_k),\\
   \chi_k\equiv 0 & \quad\text{on }\quad \operatorname{supp}   (\jmath\ast \varphi_{\ell,k}) \quad  \text{for}\quad  \ell\neq k.
\end{align*}
More precisely, we choose a
cutoff function $\chi$ independent of $k$, and a scaling parameter
$R_k\gtrsim 1$ (which will be crucial in the alternative argument presented
in appendix), and set
\[
\chi_k(x)=\chi\left(\frac{x-x_k}{R_k}\right).
\]
We recall that in view of Lemma~\ref{lem:MUK}, for $\ell\le k$ and
$t\in [0,T]$,
\[
 \operatorname{supp} \nabla
  \phi_{\ell,k}(t,\cdot)\cup \operatorname{supp}  \rho_{\ell,k} (t,\cdot)
  	\subset  \operatorname{supp}   \jmath\ast \varphi_{\ell,k}.
\]
We consider the mass localized near the bubble living at scale $k$,
\begin{equation*}
   M_{k}(t) := \|\chi_k u_k(t)\|_{L^2}^2,
\end{equation*}
and show that is bounded. 

\begin{lemma}[Localized mass estimate]\label{lem:loc-mass}
There exists $C>0$ such that for every $t\in [0,T]$ and $k\geq 1$,
\[
M_k(t)\leq C M_k(0)\lesssim 1.
\]
\end{lemma}

\begin{proof}
At $t=0$, we have 
\[
M_k(0)=\|\jmath \ast \alpha_k\|_{L^2}= \|\jmath\ast \alpha\|_{L^2}.
\]
Let us now estimate the time derivative of the localized mass.
Using the equation satisfied by $u_k$, \eqref{eq:NLSsemi}, we get
\begin{equation*}
\frac{\dd M_k}{\dd t}= 2\RE \int \chi_k^2\bar u_k\d_t u_k=\eps_k \RE
\int \chi_k^2\bar u_k i\Delta u_k
	 = 2\eps_k\IM \int \chi_k \bar u_k \nabla \chi_k\cdot
              \nabla u_k.
\end{equation*}
Given that $\nabla\chi_k=0$ on the support of $\tilde \phi_k$ and of $\varphi$, we have  
\begin{equation*}
\frac{\dd M_k}{\dd t}
    = 2\IM \int \chi_k \bar u_k \nabla \chi_k \cdot
       \( \(\eps_k\nabla u_k-i \nabla \tilde\phi_k\)u_k-\eps_k\nabla \varphi\).
 \end{equation*}
Cauchy-Schwarz and Young inequalities yield
\begin{equation*}
\frac{\dd M_k}{\dd t}
	\lesssim \|\nabla\chi_k\|_{L^{\infty}} \left(\tilde H_k+M_k\right),
      \end{equation*}
where $\tilde H_k$ is small, from
Proposition~\ref{prop:modulated}. Gronwall lemma implies that for
$t\in [0,T]$, 
\[
M_k(t)\lesssim M_k(0)+o_k(1)\lesssim M_k(0).\qedhere
\]
\end{proof}

\subsection{Lower and upper bounds on the $\dot{H}^1$ norm}

\begin{lemma}[Bounds on the local energy]\label{lem:H1}
There exist $C>0$ and $\tau\in [0,T]$ such that for every  $k\geq 1$, 
\[
\frac{1}{C}\leq  \| (\eps_k\nabla)(\chi_k  u_k)(\tau)\|_{L^2}\leq C.
\]
\end{lemma}

\begin{proof}
We develop
\[
(\eps_k\nabla)(\chi_k u_k)
	= \eps_k u_k\nabla \chi_k +   \chi_k \eps_k\nabla u_k.
\]
Let $\tilde\chi\in\mathcal{C}_c^{\infty}$ such that $\tilde\chi\geq |\nabla \chi|$ on $\operatorname{supp}(\chi)$, and such that 
\[
\tilde \chi_k=\tilde \chi\left(\frac{x-x_k}{R_k}\right)
\]
also satisfies
\begin{align*}
   \tilde\chi_k\equiv 1 &\quad\text{on }\quad \operatorname{supp} (\jmath\ast \varphi_{k,k}),\\
   \tilde\chi_k\equiv 0 &\quad\text{on }\quad \operatorname{supp} (\jmath\ast \varphi_{\ell,k}), \quad \ell\neq k.
\end{align*}
Then Lemma~\ref{lem:loc-mass} applied with $\tilde \chi_k$ instead of
$\chi_k$ yields
\[
\|\nabla \chi_ku_k(t)\|_{L^2}
	\lesssim \|\tilde\chi_k u_k(t)\|_{L^2}
	\lesssim 1.
\]
Moreover, since $\varphi$ and $\chi$ have disjoint supports, we have the inequality
\[
\| \chi_k \eps_k\nabla u_k\|_{L^2}
	\gtrsim \| \chi_k \tilde V_k u_k\|_{L^2}-\| \chi_k ((\eps_k\nabla-i\tilde V_k) u_k-\eps_k\nabla\varphi)\|_{L^2}.
\]
Now we argue like in~\cite[Lemma 5.3]{ACMA}: in view of the coupling
in \eqref{eq:bulle-univ-reg}, there 
exists $\tau\in [0,T]$ such that
\begin{equation*}
  \|a(\tau)\nabla\phi(\tau)\|_{L^2}\gtrsim 1,
\end{equation*}
hence, for every $k$,
\[
\| \chi_k \tilde V_k (\tau)\tilde a_k(\tau)\|_{L^2}=\|
a_k(\tau)\nabla\phi_k (\tau)\|_{L^2}\gtrsim 1. 
\]
Using that $|\tilde a_k|^2-|u_k|^2$ is small in $L^{m+1}$ thanks to the renormalized modulated energy estimate from Proposition~\ref{prop:modulated}, we deduce that
\[
\| \chi_k \tilde V_k u_k\|_{L^2}\gtrsim 1.
\]
Using Proposition~\ref{prop:modulated} to estimate the term $\| \chi_k
((\eps_k\nabla-i\tilde V_k) u_k-\eps_k\nabla\varphi)\|_{L^2}$
(which is controlled by the renormalized kinetic term), we deduce that
\[
\| \chi_k \eps_k\nabla u_k(\tau)\|_{L^2}\gtrsim 1.
\]

Note that  there also holds 
\[
\| \chi_k \tilde V_k(t) \tilde a_k(t)\|_{L^2}\lesssim 1,\quad \forall
t\in [0,T],
\]
hence we also know that there is also an upper bound with the same order:
\[
\| \chi_k \eps_k\nabla u_k(t)\|_{L^2}\lesssim 1,\quad \forall
t\in [0,T],
\]
hence the lemma.
\end{proof}

\subsection{Higher-order Sobolev norms}
\label{sec:higher}

In this section, we fix $0<\si<2$, and prove that the homogeneous Sobolev norm
$\dot H^\si$ of the solution $u_k$ at time $\tau$, provided by
Lemma~\ref{lem:H1},  grows like $\eps_k^{-\si}$ as $k\to
+\infty$. First, in view of Lemmas~\ref{lem:loc-mass} and
\ref{lem:H1}, we have, by interpolation, for $\si>1$,
\begin{align*}
  \begin{aligned}
  1\lesssim \left\|(\eps_k\nabla)\(\chi_ku_k(\tau)\)\right\|_{L^2}&\le
  \left\|\chi_ku_k(\tau)\right\|_{L^2}^{1-1/\si}
  \left\||\eps_k\nabla|^\si\(\chi_ku_k(\tau)\)\right\|_{L^2}^{1/\si}\\
  &\lesssim
  \left\||\eps_k\nabla|^\si\(\chi_ku_k(\tau)\)\right\|_{L^2}^{1/\si}.
  \end{aligned}
\end{align*}
For the case $0<\si<1$, we invoke the following result:
\begin{lemma}[Lemma~5.1 from \cite{ACMA}]
  There exists a constant $K$ such that, for all 
$\eps\in ]0,1]$, for all $\si\in [0,1]$, 
for all $u\in H^{1}(\R^n)$ and for all $v\in W^{1,\infty}(\R^n)$,
\begin{equation*}
\| |v|^\si u\|_{L^2}\le \| |\eps D_{x}|^\si u\|_{L^2}+\| (\eps\nabla -i
v)u\|_{L^2}^{\si} 
\| u\|_{L^2}^{1-\si}+
\eps^{\si/2} K \(1+\| \nabla v\|_{L^\infty}\)\| u\|_{L^2}.
\end{equation*}
\end{lemma}
Applying this lemma to $u=\chi_k u_k$, $v=\nabla \phi_k$, and noticing
that $\chi_k\nabla \varphi\equiv 0$ so we may invoke
Proposition~\ref{prop:modulated}, we infer, thanks to Lemma~\ref{lem:loc-mass}:
\begin{equation}\label{eq:bound-below-sob}
  \left\||\eps_k\nabla|^\si\(\chi_ku_k(\tau)\)\right\|_{L^2}\gtrsim
  1,\quad \forall \si>0.
\end{equation}

\begin{proposition}[Norm inflation]\label{prop:H^si} Let $0<\si\leq
  2$, then, with $\tau\in [0,T]$ as in Lemma~\ref{lem:H1},
\[
\| | \eps_k\nabla|^{\si} u_k(\tau)\|_{L^2}	\gtrsim 1.
\]
\end{proposition}
Roughly speaking, the above result consists in getting rid of the cutoff
function in \eqref{eq:bound-below-sob}. 
In the
case $\si=2$, this is direct from Leibniz formula:
\begin{align*}
  \eps_k^2\Delta \(\chi_ku_k\)=\eps_k^2 \chi_k \Delta u_k +
  2\eps_k^2\nabla\chi_k\cdot \nabla u_k+\eps_k^2 u_k\Delta \chi_k.
\end{align*}
The $L^2$-norm of the second term is $\O(\eps_k)$ from
Lemma~\ref{lem:H1}, and the last term is  $\O(\eps_k^2)$ from
Lemma~\ref{lem:loc-mass}, hence, in view of \eqref{eq:bound-below-sob},
\begin{equation*}
  \|\eps_k^2\Delta u_k(\tau)\|_{L^2}\ge \|\eps_k^2\chi_k\Delta
  u_k(\tau)\|_{L^2}+\mathcal{O}(\eps_k)\gtrsim 1. 
\end{equation*}

To prove Proposition~\ref{prop:H^si} in the case $0<\si<2$, we first
recall a characterization of the 
homogeneous Sobolev 
norms, based on the seminal work \cite{Besov61}, which then makes it
possible to easily extend the above Leibniz formula to the case of
fractional derivative.
\begin{lemma}\label{lem:Besov}
  Let $\si\in (0,2)$, $\si<d/2$. The following equivalence holds:
  \begin{equation*}
    \|u\|_{\dot H^{\si}}^2\approx \iint_{\R^{2d}}
      \left|u(x+y)+u(x-y)-2u(x)\right|^2\frac{\dd x \dd y}{|y|^{d+2\si}}.
  \end{equation*}
\end{lemma}
\begin{proof}
  From Plancherel equality in $x$, the quantity on the right hand side
  is equal to
  \begin{equation*}
     4\iint_{\R^{2d}}
      \left|\cos(y\cdot \xi)-1\right|^2|\hat u(\xi)|^2\frac{\dd\xi \dd y}{|y|^{d+2\si}}.
  \end{equation*}
By an homogeneous change of variable and a rotation, the integral in
$y$  is equal to
\begin{align*}
  \int_{\R^d}\left|\cos(y\cdot \xi)-1\right|^2\frac{
    \dd y}{|y|^{d+2\si}}&= |\xi|^{2\si}\int_{\R^d}\left|\cos\(z\cdot
    \frac{\xi}{|\xi|}\)-1\right|^2\frac{ 
    \dd z}{|z|^{d+2\si}}\\
&= |\xi|^{2\si}\int_{\R^d}\left|\cos\(z_1\)-1\right|^2\frac{ 
    \dd z}{|z|^{d+2\si}}.
\end{align*}
This quantity is finite as soon as $\si\in (0,2)$.
\end{proof}
Note that for $\delta>0$, the integral analyzed in the proof of the
lemma, when restricted to the region 
$\{|y|>\delta\}$, is controlled by
\begin{equation*}
  \int_{|y|>\delta}\left|\cos(y\cdot \xi)-1\right|^2|\frac{
    \dd y}{|y|^{d+2\si}}\le 4 \int_{|y|>\delta}\frac{
    \dd y}{|y|^{d+2\si}}\lesssim \delta^{-2\si},
   \end{equation*}
and so
\begin{equation*}
  \iint_{|y|>\delta}
      \left|u(x+y)+u(x-y)-2u(x)\right|^2\frac{\dd x \dd y}{|y|^{d+2\si}}\lesssim
      \delta^{-2\si}\|u\|_{L^2}^2. 
\end{equation*}
We apply these properties to the function $\chi_k u_k$: for $\delta>0$ to
be fixed later, we have
\begin{align*}
  \eps_k^{2\si} \|\chi_k u_k\|_{\dot H^{\si}}^2&\lesssim
  \(\frac{\eps_k}{\delta}\)^{2\si}\|\chi_k u_k\|_{L^2}^2 \\
+  \eps_k^{2\si}\iint_{|y|<\delta}&
      \left|(\chi_k u_k)(x+y)+(\chi_k u_k)(x-y)-2(\chi _ku_k)(x)\right|^2\frac{\dd x\dd y}{|y|^{d+2\si}}.
\end{align*}
We leave out the index $k$ in order to lighten notations. 
For the last integral, we use a discrete form of Leibniz formula:
\begin{align*}
  (\chi u)(x+y)&+(\chi u)(x-y)-2(\chi u)(x)=\\
&\quad\chi(x)\( u(x+y)+u(x-y)-2u(x)\)\\
&+\(\chi(x+y)-\chi(x)\)\(u(x+y)-u(x)\) \\
&+ \(\chi(x)-\chi(x-y)\)\(u(x)-u(x-y)\) \\
&+u(x)\(\chi(x+y)+\chi(x-y)-2\chi(x)\). 
\end{align*}
The integral corresponding to the first term on the right hand side is
obviously estimated by
\begin{equation*}
  \|\chi\|_{L^\infty}^2 \|u\|_{\dot H^{\si}}^2.
\end{equation*}
The second and third terms are similar, and the corresponding
integrals are actually equal, through the change of variable $y\mapsto
-y$. We choose $\delta>0$ (independent of $k$) such that 
\begin{equation*}
  \forall \ell\not = k,\quad \(\operatorname{supp}\chi
  +B(0,3\delta)\)\cap \operatorname{supp}a_{\ell,k} = \emptyset,
\end{equation*}
and we pick another cut-off function $\chi_1\in
C_0^\infty(\R^d;[0,1])$ such that $\chi_1\equiv 1$ on $\operatorname{supp}\chi
  +B(0,\delta)$, and 
  \begin{equation*}
    \operatorname{supp}\chi_1\subset \operatorname{supp}\chi
  +B(0,3\delta).
  \end{equation*}
We thus have
\begin{multline*}
  \(\chi(x+y)-\chi(x)\)\(u(x+y)-u(x)\)\\ =
  \(\chi(x+y)-\chi(x)\)\((\chi_1u)(x+y)-(\chi_1 u)(x)\).
\end{multline*}
We infer
\begin{align*}
  \iint_{|y|<\delta}&
      \left|\(\chi(x+y)-\chi(x)\)\(u(x+y)-u(x)\)
  \right|^2\frac{\dd x\dd y}{|y|^{d+2\si}} \\
& \lesssim \|\nabla
  \chi\|_{L^\infty}^2 \iint_{|y|<\delta}
      \left|(\chi_1 u)(x+y)-(\chi_1 u)(x)
  \right|^2\frac{\dd x\dd y}{|y|^{d+2\si-2}}\\
& \lesssim \|\nabla
  \chi\|_{L^\infty}^2 \|\chi_1 u\|_{\dot H^{\max(0,\si-1)}}^2,
\end{align*}
where we have used the more standard characterization of the $\dot
H^{\si-1}$-norm 
(note that $\si-1<1$)
analogous to the one given in Lemma~\ref{lem:Besov}. 

Finally, the last integral is equal to 
\begin{align*}
    \iint_{|y|<\delta}|(\chi_1
  u)(x)|^2\left|\chi(x+y)+\chi(x-y)-2\chi(x)\right|^2
  \frac{\dd x\dd y}{|y|^{d+2\si}} \lesssim \|\chi\|_{W^{2,\infty}}^2 \|\chi_1
  u\|_{L^2}^2,
\end{align*}
since, for $\delta\le 1$, 
\begin{equation*}
  \int_{|y|<\delta}  \frac{\dd y}{|y|^{d+2\si}} \lesssim \int_0^1
  \frac{r^{d-1}}{r^{d-2\si}}\dd r<\infty. 
\end{equation*}
We conclude:
\begin{equation*}
  1\lesssim  \eps^{2\si} \|\chi u(\tau)\|_{\dot H^{\si}}^2\lesssim
 \eps^{2\si} \|u(\tau)\|_{\dot H^{\si}}^2 + \underbrace{\eps^{2\si}\(\|\chi_1
 u(\tau)\|_{L^2}^2 + \|\chi_1 u(\tau)\|_{\dot H^{\max(0,\si-1)}}^2\)}_{=\O(\eps^{2\si})},
\end{equation*}
hence 
\begin{equation*}
  \eps^{2\si} \|u\|_{\dot H^{\si}}^2 \gtrsim 1,\quad \forall \si\in (0,2).
\end{equation*}

\subsection{Conclusion}
\label{sec:conclusion}

We can now go back to the original function
\color{black}
\[
\psi_k(t,x)=h_k^{s-\frac d2}\frac{1}{|\log(h_k)|} u_k\left(\frac{t}{\eps_k h_k^2},\frac{x}{h_k}\right).
\]
\color{black}
This function satisfies 
\[
\||\nabla|^{\si}\psi_k(\eps_kh_k^2 t)\|_{L^2}=h_k^{s-\si}|\log(h_k)|^{-1}\||\nabla|^{\si} u_k(t)\|_{L^2}.
\]
Let $t_k= \eps_kh_k^2\tau$, where $\tau$ stems from
Lemma~\ref{lem:H1}: $t_k\to 0$ as $k\to +\infty$. 
In view of Proposition~\ref{prop:H^si}, we know that when $0<\si<2$,
\[
\||\nabla|^{\si}\psi_k(t_k)\|_{L^2}
	\gtrsim h_k^{s-\si}\eps_k^{-\si}|\log(h_k)|^{-1}.
\]
Given the choice $\eps_k=h_k^{m(s_c-s)}|\log(h_k)|^{m}$, the lower bound goes to infinity as $k$ goes to infinity as soon as
\[
s<\si(1+m(s_c-s)).
\]

Finally, the density of initial data $f_0$ in the pathological set is a direct consequence of~\cite[Proposition 2.10]{CaGa23}.

\subsection*{Acknowledgements} The authors are grateful to Nikolay
Tzvetkov for suggesting to reconsider the result of \cite{ACMA} in the light of
\cite{CaGa23}, to Thomas Alazard for the arguments presented in
Section~\ref{sec:higher}, and also to Nicolas Camps for ideas leading to Appendix~\ref{sec:alternative}. 

\appendix

\section{Alternative proof of localized estimates}
\label{sec:alternative}

In this appendix, we provide an alternative proof of
Proposition~\ref{prop:H^si}. It has the drawback of being longer than
in Section~\ref{sec:higher}, but the advantage of being more flexible, and easily
adapted to the case of geometries different from $\R^d$, typically
$\T^d$. 

\subsection{Tuning the size of cutoff functions}

We provide a proof that relies on the sizes of the cutoff $R_k$
around the bubble $k$, as introduced in Section~\ref{sec:local}. We
first show that up to some conditions on  $R_k$, norm
inflation occurs. In a second step, we prove that such conditions on
$R_k$ can be met
even on $\T^d$. 

\begin{lemma}[Size of the cutoff functions] Let $0<\si\leq 2$. There exists $C,C'>0$ such that for every $t\lesssim 1$ and $k\geq 1$, assuming that
\[
\begin{cases}
R_k^{\si}\geq C'h_k^{-s}\eps_k^{\si} &\quad \text{if} \quad {0<\si <1}, 
\\
R_k^{\si-1}\geq C'h_k^{1-s}\eps_k^{\si}h_{k-1}^{(m+1)(s_{\rm sob}-s)} &\quad \text{if}\quad  {1<\si < 2} ,
\end{cases}
\]
then there holds
\[
\| | \eps_k\nabla|^{\si} u_k\|_{L^2}	\gtrsim 1.
\]
\end{lemma}
The conclusion of the above lemma corresponds to that of
Proposition~\ref{prop:H^si}, and from that the proof of norm inflation
is inferred like in
Section~\ref{sec:conclusion}. 
\begin{proof}
  Recall that in view of \eqref{eq:bound-below-sob},
  \begin{equation*}
  \left\||\eps_k\nabla|^\si\(\chi_ku_k(\tau)\)\right\|_{L^2}\gtrsim
  1,\quad \forall \si>0.
\end{equation*}
We first treat the case $0<\si<1$, and write
\begin{equation*}
| \nabla|^{\si}(\chi_k u_k)
	=[| \nabla|^{\si}, \chi_k] u_k +\chi_k |\nabla|^{\si}u_k.
\end{equation*}
Using the commutator estimate recalled in Lemma~\ref{lem:commutator} below,
\[
\|[| \nabla|^{\si}, \chi_k]  u_k\|_{L^2}
	\lesssim \frac{\| u_k\|_{L^2}}{R_k^{\si}}.
\]
Given the non-localized estimate~\eqref{eq:L2} on  the $L^2$-norm of
$u_k$, the upper bound is less than $\frac{\eps_k^{-\si}}{2 C}$ as
soon as for some $C'>0$, we have
\[
R_k^{\si }\geq C' h_k^{-s}\eps_k^{\si}.
\] 
In this case we conclude that
\[
\|\chi_k | \nabla|^{\si}u_k(\tau)\|_{L^2}
	\gtrsim  \eps_k^{-\sigma}.
\]
In the case $1<\si<2$,
we expand
\begin{align*}
| \nabla|^{\si}(\chi_k u_k)
	&=| \nabla|^{\si-1}(\nabla \chi_k u_k+\chi_k\nabla u_k)\\
	&=| \nabla|^{\si-1}(\nabla \chi_k u_k) +[| \nabla|^{\si-1}, \chi_k] \nabla u_k +\chi_k |\nabla|^{\si}u_k.
\end{align*}
Using interpolation, we have
\[
\| | \nabla|^{\si-1}(\nabla \chi_k u_k)\|_{L^2}
	\lesssim \| \nabla \chi_k u_k\|_{L^2}^{2-\si}
\| \nabla(\nabla \chi_k u_k)\|_{L^2}^{\si-1}.
\]
Thanks to Lemmas~\ref{lem:loc-mass} and~\ref{lem:H1} applied to the cutoff function $\nabla \chi_k$, this term is uniformly bounded in $k$ when $t\in[0,T]$.
Concerning the second term in the expansion, we apply the commutator estimate from Lemma~\ref{lem:commutator} with $\alpha=\si-1$ and get that
\[
\|[|\nabla|^{\si-1},\nabla\chi_k]\nabla u_k\|_{L^2}
	\lesssim \frac{\|\nabla u_k\|_{L^2}}{R_k^{\si-1}}.
\]
As a consequence of the energy estimate~\eqref{eq:H1}, if 
\[
R_k^{\si-1}\geq C' h_k^{1-s}\eps_k^{\si}h_{k-1}^{(m+1)(s_{\rm sob}-s)},
\]
for some $C'>0$ large enough, we get that
\[
\|\chi_k |\nabla|^{\si}u_k(\tau)\|_{L^2} \gtrsim \eps_k^{-\si}.
\]
In the case $\si=2$, the argument presented in
Section~\ref{sec:higher}, based on Leibniz formula and localized
estimates in $L^2$ and $\dot H^1$, respectively, needs no modification.
\end{proof}

We now check that the condition on the size of $R_k$ from the above
lemma can be realized with a suitable choice of positions
$\frac{x_k}{h_k}$ for the bubbles.  

\begin{lemma}[Fitting the bubbles on the torus]
One can fix $x_k$ such that there exists $R_k$ satisfying $\overline{B}(\frac{x_k}{h_k},R_k) \cap \operatorname{supp} (\jmath\ast \varphi_{\ell,k})=\emptyset$ for $\ell\neq k$ and
\[
\begin{cases}
R_k^{\si}\geq C'h_k^{-s}\eps_k^{\si} &\quad \text{if} \quad {0<\si <1}, 
\\
R_k^{\si-1}\geq C'h_k^{1-s}\eps_k^{\si}h_{k-1}^{(m+1)(s_{\rm sob}-s)} &\quad \text{if}\quad  {1<\si < 2}.
\end{cases}
\]
\end{lemma}
Under these conditions, it is possible to construct a cutoff $\chi_k$ of radius $R_k$ around the $k$-th bubble located around position $\frac{x_k}{h_k}$ in the rescaled variables for $u_k$.

\begin{proof}
We recall that in the original variables, we have
\[
\psi_k(t,x)=h_k^{s-\frac d2}|\log(h_k)|^{-1}u\left(\frac{t}{\eps_k h_k^2},\frac{x}{h_k}\right).
\]
In the original variables, the cutoff $\chi_k$ corresponds to a cutoff of size $h_k R_k$ around position $x_k$. Moreover, $h_k=e^{-M^k}$ for some parameter $M>1$. As a consequence, it is possible to fit all the bubbles in the torus if there exists $\delta>0$ such that
\[
h_kR_k\lesssim h_k^{\delta}.
\]

In the case $0<\si<1$, we combine this condition with the lower bound and $R_k$ and we deduce that this is possible if
\[
h_k^{\sigma(\delta-1)}\gtrsim  c \eps_k^{\si}h_k^{-s}.
\]
Given the formula $\eps_k=h_k^{m(s_c-s)}|\log(h_k)|^{m}$, it is enough that
\[
\sigma(\delta-1)<\sigma m(s_c-s)-s,
\]
that is,
\[\sigma(1+m(s_c-s)-\delta)>s.
\]
 We check that when 
\[
\sigma(1+m(s_c-s))>s,
\]
then it is always possible to chose $\delta>0$ small enough so that this condition is satisfied.

In the case $1<\si\leq 2$, this condition is compatible with the lower bound on $R_k$ if
\[
h_k^{(\sigma-1)(\delta-1)}\gtrsim c \eps_k^{\si}h_k^{1-s} h_{k-1}^{(m+1)(s_{\rm sob}-s)}.
\]
Given the formula for $\eps_k$, it is enough that
\[
(\sigma-1)(\delta-1)<\sigma m(s_c-s)+1-s,
\]
that is
\[
\sigma(1+m(s_c-s)-\delta)>s-\delta.
\]
We check that when 
\[
\sigma(1+m(s_c-s))>s,
\] then it is always possible to chose $\delta>0$ small enough so that this condition is satisfied.
\end{proof}

\subsection{Commutator estimate}

\begin{lemma}[Commutator estimate]\label{lem:commutator} For every $0<\alpha<1$, there holds
\[
\|[|\nabla|^{\alpha},\chi_k]f\|_{L^2}
	\lesssim \frac{\|\chi\|_{W^{1,\infty}}}{R_k^{\alpha}}\|f\|_{L^2}.
\]
\end{lemma}

We give an elementary proof of this result, which is
adapted from~\cite[Lemma~D.1]{GLPR18}. 
\begin{proof}
The operator $|\nabla|^{\alpha}$ is the convolution operator with the tempered distribution $\kappa=\mathcal{F}^{-1}(|\xi|^{\alpha})$, which is homogeneous of degree $-\alpha-d$ and even. We deduce that there exists a universal constant $c$ such that
\[
|\nabla|^{\alpha}f(x)
	=\kappa\ast f(x)=c\int_{\R^d}\frac{f(y)-f(x)}{|y-x|^{\alpha+d}}\dd y.
\] 
As a consequence,
\[
[|\nabla|^{\alpha},\chi_k]f
	=c\int_{\R^d}\frac{\chi_k(y)-\chi_k(x)}{|y-x|^{\alpha+d}}f(y)\dd y.
\]
We split the integral between the parts $|y-x|\leq R_k$ and $|y-x|>R_k$. First, using the fact that $\chi_k$ is Lipschitz continuous, $\frac{|\chi_k(y)-\chi_k(x)|}{|y-x|^{\alpha+d}} \leq \frac{\|\nabla \chi_k\|_{L^{\infty}}}{|y-x|^{\alpha+d-1}}$ so that
\begin{align*}
\left\|\int_{|y-x|\leq R_k}\frac{\chi_k(y)-\chi_k(x)}{|y-x|^{\alpha+d}}f(y)\dd y\right\|_{L^2_x}
  &\lesssim \|\nabla \chi_k\|_{L^{\infty}}
    \left\|\frac{\mathbf{1}_{|x|\leq R_k}}{|x|^{\alpha+d-1}}\ast f\right\|_{L^2}\\
  &\lesssim\frac{\|\nabla \chi\|_{L^{\infty}}}{R_k}
    \left\|\frac{\mathbf{1}_{|x|\leq R_k}}{|x|^{\alpha+d-1}}\right\|_{L^1}\| f\|_{L^2}\\
  &\lesssim\frac{\|\nabla \chi\|_{L^{\infty}}}{R_k^{\alpha}}
    \left\|\frac{\mathbf{1}_{|x|\leq 1}}{|x|^{\alpha+d-1}}\right\|_{L^1}\| f\|_{L^2}.
\end{align*}
The term $\|\frac{\mathbf{1}_{|x|\leq 1}}{|x|^{\alpha+d-1}}\|_{L^1}$ is finite since $\alpha+d-1<d$.
Then, we have
\begin{align*}
\left\|\int_{|y-x|> R_k}\frac{\chi_k(y)-\chi_k(x)}{|y-x|^{\alpha+d}}f(y)\dd y\right\|_{L^2_x}
  &\lesssim \| \chi_k\|_{L^{\infty}}
    \left\|\frac{\mathbf{1}_{|x|> R_k}}{|x|^{\alpha+d}}\ast f\right\|_{L^2}\\
  &\lesssim \| \chi\|_{L^{\infty}}
    \left\|\frac{\mathbf{1}_{|x|> R_k}}{|x|^{\alpha+d}}\right\|_{L^1}\| f\|_{L^2}\\
  &\lesssim \frac{\| \chi\|_{L^{\infty}}}{R_k^{\alpha}}
    \left\|\frac{\mathbf{1}_{|x|> 1}}{|x|^{\alpha+d}}\right\|_{L^1}\| f\|_{L^2}.
\end{align*}
The $L^1$-norm is finite since $\alpha+d>d$.
\end{proof}

\color{black}
\section{Proof of Theorem~\ref{thm:0}}\label{sec:th0}

As mentioned in the introduction, the proof of Theorem~\ref{thm:main}
is readily adapted in order to prove Theorem~\ref{thm:0}. Let $f_0\in
H^s(\R^d)$: first, there exists a sequence $g_{0,k}\in \mathcal C_c^\infty
(R^d)$ such that
\begin{equation*}
  \|f_0-g_{0,k}\|_{H^s}\Tend k \infty 0. 
\end{equation*}
We then complement the initial datum $g_{0,k}$ with a single bubble,
\begin{equation*}
  f_{0,k} = g_{0,k} + \varphi_k,
\end{equation*}
where $\varphi_k$ is like in the rest of the paper, as introduced in
Section~\ref{sec:def-set}, with the requirement that for any $k$, the
supports of $g_{0,k}$ and $\varphi_k$ are disjoint. In view of the
logarithmic factor $|\log(h_k)|^{-1}$ in \eqref{eq:def-bubble},
\begin{equation*}
  \|f_0-f_{0,k}\|_{H^s}\le
  \|f_0-g_{0,k}\|_{H^s}+\|\varphi_k\|_{H^k}\Tend k \infty 0.  
\end{equation*}
The proof of
Theorem~\ref{thm:main} can then be repeated, by rescaling the unknown
function as in Section~\ref{sec:rescale}: the initial datum $u_{0,k}$
is the sum of the profile $\alpha$ (possibly shifted in space), and a unique
``low mode'' stemming from $g_{0,k}$. The modulated energy estimate
from Section~\ref{sec:renormalized} remains valid (note that the low
mode stemming from $g_{0,k}$ must be incorporated into the
renormalized modulated energy, in the spirit of
the proof of \cite[Theorem1.33]{Tz19}, and as in
\cite{SuTz20,CaGa23}), and we conclude like in Section~\ref{sec:local}.

\color{black}
\bibliographystyle{amsplain}
\bibliography{loss}

\providecommand{\bysame}{\leavevmode\hbox to3em{\hrulefill}\thinspace}
\providecommand{\MR}{\relax\ifhmode\unskip\space\fi MR }
\providecommand{\MRhref}[2]{%
  \href{http://www.ams.org/mathscinet-getitem?mr=#1}{#2}
}
\providecommand{\href}[2]{#2}
\begin{thebibliography}{10}

\bibitem{ACMA}
Thomas Alazard and R\'{e}mi Carles, \emph{Loss of regularity for supercritical
  nonlinear {S}chr\"{o}dinger equations}, Math. Ann. \textbf{343} (2009),
  no.~2, 397--420. \MR{2461259}

\bibitem{ACIHP}
\bysame, \emph{W{KB} analysis for the {G}ross-{P}itaevskii equation with
  non-trivial boundary conditions at infinity}, Ann. Inst. H. Poincar\'{e} C
  Anal. Non Lin\'{e}aire \textbf{26} (2009), no.~3, 959--977. \MR{2526411}

\bibitem{BeOhPo15}
\'{A}rp\'{a}d B\'{e}nyi, Tadahiro Oh, and Oana Pocovnicu, \emph{On the
  probabilistic {C}auchy theory of the cubic nonlinear {S}chr\"{o}dinger
  equation on {$\Bbb{R}^d$}, {$d\geq3$}}, Trans. Amer. Math. Soc. Ser. B
  \textbf{2} (2015), 1--50. \MR{3350022}

\bibitem{BeOhPo19}
\bysame, \emph{Higher order expansions for the probabilistic local {C}auchy
  theory of the cubic nonlinear {S}chr\"{o}dinger equation on {$\Bbb{R}^3$}},
  Trans. Amer. Math. Soc. Ser. B \textbf{6} (2019), 114--160. \MR{3919013}

\bibitem{Besov61}
O.~V. Besov, \emph{Investigation of a class of function spaces in connection
  with imbedding and extension theorems}, Trudy. Mat. Inst. Steklov.
  \textbf{60} (1961), 42--81. \MR{0133675}

\bibitem{Bou94}
Jean Bourgain, \emph{Periodic nonlinear {S}chr\"{o}dinger equation and
  invariant measures}, Comm. Math. Phys. \textbf{166} (1994), no.~1, 1--26.
  \MR{1309539}

\bibitem{Bou96}
\bysame, \emph{Invariant measures for the {$2$}{D}-defocusing nonlinear
  {S}chr\"{o}dinger equation}, Comm. Math. Phys. \textbf{176} (1996), no.~2,
  421--445. \MR{1374420}

\bibitem{BGTENS}
Nicolas Burq, Patrick G\'{e}rard, and Nikolay Tzvetkov, \emph{Multilinear
  eigenfunction estimates and global existence for the three dimensional
  nonlinear {S}chr\"{o}dinger equations}, Ann. Sci. \'{E}cole Norm. Sup. (4)
  \textbf{38} (2005), no.~2, 255--301. \MR{2144988}

\bibitem{BuThTz13}
Nicolas Burq, Laurent Thomann, and Nikolay Tzvetkov, \emph{Long time dynamics
  for the one dimensional non linear {S}chr\"{o}dinger equation}, Ann. Inst.
  Fourier (Grenoble) \textbf{63} (2013), no.~6, 2137--2198. \MR{3237443}

\bibitem{BT08a}
Nicolas Burq and Nikolay Tzvetkov, \emph{Random data {C}auchy theory for
  supercritical wave equations. {I}. {L}ocal theory}, Invent. Math.
  \textbf{173} (2008), no.~3, 449--475. \MR{2425133}

\bibitem{BT08b}
\bysame, \emph{Random data {C}auchy theory for supercritical wave equations.
  {II}. {A} global existence result}, Invent. Math. \textbf{173} (2008), no.~3,
  477--496. \MR{2425134}

\bibitem{CaGa23}
Nicolas Camps and Louise Gassot, \emph{Pathological set of initial data for
  scaling-supercritical nonlinear {Schr\"odinger} equations}, Int. Math. Res.
  Not. IMRN (2023), To appear. Archived at
  \url{https://arxiv.org/abs/2203.04840}.

\bibitem{CaARMA}
R\'{e}mi Carles, \emph{Geometric optics and instability for semi-classical
  {S}chr\"{o}dinger equations}, Arch. Ration. Mech. Anal. \textbf{183} (2007),
  no.~3, 525--553. \MR{2278414}

\bibitem{CW90}
Thierry Cazenave and Fred~B. Weissler, \emph{The {C}auchy problem for the
  critical nonlinear {S}chr\"{o}dinger equation in {$H^s$}}, Nonlinear Anal.
  \textbf{14} (1990), no.~10, 807--836. \MR{1055532}

\bibitem{CCT2}
Michael Christ, James Colliander, and Terence Tao, \emph{Ill-posedness for
  nonlinear {S}chr\"odinger and wave equations},
  \url{https://arxiv.org/abs/math.AP/0311048}.

\bibitem{DoLuMe19}
Benjamin Dodson, Jonas L\"{u}hrmann, and Dana Mendelson, \emph{Almost sure
  local well-posedness and scattering for the 4{D} cubic nonlinear
  {S}chr\"{o}dinger equation}, Adv. Math. \textbf{347} (2019), 619--676.
  \MR{3920835}

\bibitem{GLPR18}
Patrick G\'{e}rard, Enno Lenzmann, Oana Pocovnicu, and Pierre Rapha\"{e}l,
  \emph{A two-soliton with transient turbulent regime for the cubic half-wave
  equation on the real line}, Ann. PDE \textbf{4} (2018), no.~1, Paper No. 7,
  166. \MR{3747579}

\bibitem{GV85c}
Jean Ginibre and Giorgio Velo, \emph{The global {C}auchy problem for the
  nonlinear {S}chr\"{o}dinger equation revisited}, Ann. Inst. H. Poincar\'{e}
  Anal. Non Lin\'{e}aire \textbf{2} (1985), no.~4, 309--327. \MR{801582}

\bibitem{Grenier98}
Emmanuel Grenier, \emph{Semiclassical limit of the nonlinear {S}chr\"{o}dinger
  equation in small time}, Proc. Amer. Math. Soc. \textbf{126} (1998), no.~2,
  523--530. \MR{1425123}

\bibitem{KaPo88}
Tosio Kato and Gustavo Ponce, \emph{Commutator estimates and the {E}uler and
  {N}avier-{S}tokes equations}, Comm. Pure Appl. Math. \textbf{41} (1988),
  no.~7, 891--907. \MR{951744}

\bibitem{KPV01}
Carlos~E. Kenig, Gustavo Ponce, and Luis Vega, \emph{On the ill-posedness of
  some canonical dispersive equations}, Duke Math. J. \textbf{106} (2001),
  no.~3, 617--633. \MR{1813239}

\bibitem{KiMuVi19}
Rowan Killip, Jason Murphy, and Monica Visan, \emph{Almost sure scattering for
  the energy-critical {NLS} with radial data below {$H^1(\Bbb R^4)$}}, Comm.
  Partial Differential Equations \textbf{44} (2019), no.~1, 51--71.
  \MR{3933623}

\bibitem{Lebeau05}
Gilles Lebeau, \emph{Perte de r\'{e}gularit\'{e} pour les \'{e}quations d'ondes
  sur-critiques}, Bull. Soc. Math. France \textbf{133} (2005), no.~1, 145--157.
  \MR{2145023}

\bibitem{Majda}
Andrew Majda, \emph{Compressible fluid flow and systems of conservation laws in
  several space variables}, Applied Mathematical Sciences, vol.~53,
  Springer-Verlag, New York, 1984. \MR{748308}

\bibitem{MUK86}
Tetu Makino, Seiji Ukai, and Shuichi Kawashima, \emph{Sur la solution \`a
  support compact de l'\'{e}quations d'{E}uler compressible}, Japan J. Appl.
  Math. \textbf{3} (1986), no.~2, 249--257. \MR{899222}

\bibitem{MRRS22}
Frank Merle, Pierre Rapha\"{e}l, Igor Rodnianski, and Jeremie Szeftel, \emph{On
  blow up for the energy super critical defocusing nonlinear {S}chr\"{o}dinger
  equations}, Invent. Math. \textbf{227} (2022), no.~1, 247--413. \MR{4359478}

\bibitem{Spitz23}
Martin Spitz, \emph{Almost sure local wellposedness and scattering for the
  energy-critical cubic nonlinear {S}chr\"{o}dinger equation with supercritical
  data}, Nonlinear Anal. \textbf{229} (2023), Paper No. 113204, 33.
  \MR{4528586}

\bibitem{SuTz20}
Chenmin Sun and Nikolay Tzvetkov, \emph{Concerning the pathological set in the
  context of probabilistic well-posedness}, C. R. Math. Acad. Sci. Paris
  \textbf{358} (2020), no.~9-10, 989--999. \MR{4196770}

\bibitem{SuTz20a}
\bysame, \emph{New examples of probabilistic well-posedness for nonlinear wave
  equations}, J. Funct. Anal. \textbf{278} (2020), no.~2, 108322, 47.
  \MR{4030277}

\bibitem{ThomannAnalytic}
Laurent Thomann, \emph{Instabilities for supercritical {S}chr\"{o}dinger
  equations in analytic manifolds}, J. Differential Equations \textbf{245}
  (2008), no.~1, 249--280. \MR{2422717}

\bibitem{Tz19}
Nikolay Tzvetkov, \emph{Random data wave equations}, Singular random dynamics,
  Lecture Notes in Math., vol. 2253, Springer, Cham, [2019] \copyright 2019,
  pp.~221--313. \MR{3971360}

\bibitem{Xia21}
Bo~Xia, \emph{Generic illposedness for wave equation of power type on
  three-dimensional torus}, Int. Math. Res. Not. IMRN (2021), no.~20,
  15533--15554. \MR{4329875}

\bibitem{Xia22}
Bo~Xia and Deng Zhang, \emph{{Ill-posedness of quintic fourth order
  Schr{\"o}dinger equation}}, archived at
  \url{https://arxiv.org/abs/2202.03020}, 2022.

\end{thebibliography}

\end{document}